\documentclass{amsart}
\usepackage{amsmath, amsthm, amsbsy, amssymb, amsfonts, verbatim,array}
\newtheorem{Theorem}{Theorem}[section]
\newtheorem{Lemma}[Theorem]{Lemma}

\newtheorem{Corollary}[Theorem]{Corollary}
\newtheorem{Definition}{Definition}

\newtheorem{Remark}[Theorem]{Remark}

\def\newsection{ \separate
  \refstepcounter{subsection}  
 {\large\bf \thesubsection\kern.3em}}
 \makeatletter
\@addtoreset{equation}{subsection}

\makeatother \theoremstyle{plain}

\title{Real forms of complex surfaces of constant mean curvature}
\author{Shimpei Kobayashi}
 \address{Graduate School of Science and Technology \\ Hirosaki University \\
 Bunkyocho 3 Aomori 036-8561 Japan}
 \email{shimpei@cc.hirosaki-u.ac.jp}
\date{}

\begin{document}
\maketitle
\begin{abstract} 
 It is known that complex constant mean curvature  ({\sc
 CMC} for short) immersions in $\mathbb C^3$ are natural
 complexifications of {\sc CMC}-immersions in $\mathbb R^3$.
 In this paper, conversely we consider {\it real form surfaces} of a
 complex {\sc CMC}-immersion, which are defined from real forms
 of the twisted $\mathfrak{sl}(2, \mathbb C)$ loop algebra $\Lambda
 \mathfrak{sl}(2, \mathbb C)_\sigma$, and classify all such surfaces
 according to the classification of real forms of $\Lambda
 \mathfrak{sl}(2, \mathbb C)_\sigma$. There are seven classes of surfaces,
 which are called {\it integrable surfaces}, and all integrable surfaces
 will be characterized by the (Lorentz) harmonicities of their Gau{\ss}
 maps into the symmetric spaces $S^2$, $H^2$, $S^{1,1}$ or the
 $4$-symmetric space $SL(2, \mathbb C)/U(1)$. We also give a
 unification to all integrable surfaces via the generalized
 Weierstra{\ss} type representation.
\end{abstract}

\section{Introduction}
 The goal of this paper is to give a unified theory for {\it integrable
 surfaces} using the real forms of complex extended framings of complex
 {\sc CMC}-immersions and the generalized Weierstra{\ss} type
 representation for complex {\sc CMC}-immersions. 

 It is classically  known that {\sc CMC} surfaces with nonzero mean
 curvature, or equivalently constant positive Gau{\ss}ian curvature
 ({\sc CPC} for short) surfaces as parallel surfaces,
 and constant negative Gau{\ss}ian curvature
 ({\sc CNC} for short) surfaces in $\mathbb R^3$ are characterized by
 the transformations of real (or complex) tangential line congruences between
 surfaces with special properties, which are  commonly called ``(Bianchi)
 B\"{a}cklund transformations''. In modern terminology, such
 classes of surfaces are characterized by the (Lorentz) harmonicities of
 their Gau{\ss} maps, and they are equivalent to the
 existence of families of flat connections on $\mathcal M \times SO(3)$, where
 $\mathcal M$ is $\mathbb C$ for {\sc CMC}-immersions or $\mathbb R^{1,1}$
 for {\sc CNC}-immersions, see \cite{DPW} and \cite{MS:PSsurfaces}. 
 Spacelike or timelike constant positive or negative Gau{\ss}ian
 curvature surfaces in $\mathbb R^{1,2}$ are less known, however,
 they are also characterized by the (Lorentz) harmonicities of their Gau{\ss}
 maps, or equivalently, the existence of
 families of flat connections on $\mathcal M \times
 SO(2,1)$, where $\mathcal M$ is $\mathbb C$ for spacelike {\sc
 CNC}-immersions, or equivalently spacelike {\sc CMC}-immersions,
 and timelike {\sc CNC}-immersions, or $\mathbb R^{1, 1}$ for spacelike {\sc
 CPC}-immersions and timelike {\sc CPC}-immersions, or equivalently
 timelike {\sc CMC}-immersions, 
 see \cite{DIT:Timelike} and \cite{Klotz:Harm-Mink}.  

 On the one hand, to classify all {\sc CMC}-cylinders in $\mathbb R^3$,
 in \cite{DK:cyl} we gave a natural complexification of the extended
 framing, a moving frame with spectral parameter and an element in the
 $SU(2)$ loop group, of a {\sc CMC}-immersion, which is called the {\it
 complex extended framing}.  
 Moreover in \cite{DKP:Complex}, we introduced holomorphic
 immersions in $\mathbb C^3$ associated with the complex extended
 framings and a natural definition of the complex mean curvature for a
 holomorphic immersion. Then a holomorphic immersion with
 complex constant mean curvature $H \in \mathbb C$ is naturally called
 the {\it complex {\sc CMC}-immersion}.
 Similar to the real case, a holomorphic immersion with complex constant Gau{\ss}
 curvature $K \in \mathbb C^*$ ({\sc CGC} for short) is obtained as the parallel
 immersion of a complex {\sc CMC}-immersion with nonzero complex
 constant mean curvature $H \in \mathbb C^*$.

 In this paper, we shall interpret those complex {\sc CGC}-immersions, or
 equivalently {\sc CMC}-immersions by the parallel immersions,
 as {\it complexifications} for the surfaces discussed above. These
 real surfaces are then obtained by the real form surfaces of a complex {\sc
 CGC}-immersion, which are defined from the real forms of
 the Maurer-Cartan form of the complex extended framing of a complex {\sc
 CMC}-immersion. It is known that the twice central extensions of a loop algebra $\Lambda
 \mathfrak{sl}(2, \mathbb C)_{\sigma}$ is a twisted affine Kac-Moody Lie
 algebra of $A_1^{(1)}$ type. The classification of the real forms of affine
 Kac-Moody Lie algebras was given in \cite{BBBR:almostsplit} and
 \cite{BR:almostcompact}. It follows that the classification of real forms of $\Lambda
 \mathfrak{sl}(2, \mathbb C)_\sigma$ is also given. In particular, the real forms of
 $\Lambda \mathfrak{sl}(2, \mathbb C)_{\sigma}$ consist of seven classes: 
 three classes are called the {\it almost split} and the other four
 classes are called the {\it almost compact}, according to the types of
 the semi-linear involutions of the real forms. 
 Thus there are seven classes of surfaces, which are called {\it integrable
 surfaces}, according to the classification of the real forms of
 $\Lambda \mathfrak{sl} (2, \mathbb C)_\sigma$. Spacelike or timelike {\sc CMC} or {\sc CGC}
 surfaces in $\mathbb R^3$ or $\mathbb R^{1,2}$ form the six classes of
 integrable surfaces, and {\sc CMC} surfaces with mean curvature $|H| <1$
 in $H^3$ form the last class of integrable surfaces (Theorem
 \ref{thm:compactsplit} and Corollary \ref{coro:compactsplit}).
 Moreover, all integrable surfaces are characterized by (Lorentz) harmonicities of
 their Gau{\ss} maps, which are maps into symmetric spaces $S^2$, $H^2$,
 $S^{1,1}$ and the $4$-symmetric space $SL(2, \mathbb C)/U(1)$
 respectively, Theorem \ref{thm:Gaussmap}.

 The generalized Weierstra{\ss} type representation for complex {\sc
 CMC}-immersions is a procedure to construct complex {\sc
 CMC}-immersions in $\mathbb C^3$, see Section \ref{subsc:DPW} for more
 details: {\bf 1.} Define pairs of holomorphic
 potentials, which are pairs of holomorphic 1-forms $\check \eta =
 (\eta, \tau)$ with $\eta = \sum_{j \geq -1}^{\infty} \eta_j \lambda^j$
 and $\tau = \sum_{-\infty}^{j \leq 1} \tau_j \lambda^j$. Here $\lambda$
 is the complex parameter, the so-called ``spectral parameter'', $\eta_j$ and $\tau_j$ are diagonal
 (resp. off-diagonal) holomorphic 1-forms
 depending only on one complex variable  if $j$ is even (resp. $j$ is
 odd). {\bf 2.} Solve the pair of ODE's $d (C, L) = (C, L) \check \eta$
 with some initial condition $(C(z_*), L(w_*))$, and perform the
 generalized Iwasawa decomposition, Theorem
 \ref{doublesplitting},
 for $(C, L)$, giving $ (C, L)= (F, F)(V_+, V_{-})$. 
 It is known that $F \cdot l$ is the complex extended framing of some complex 
 {\sc CMC}-immersion Theorem \ref{thm:DKP-Extendedframings},
 where $l$ is some $\lambda$-independent
 diagonal matrix.  {\bf 3.} Form a complex {\sc CMC}-immersion 
 by the Sym formula $\varPsi$  via the complex extended framing 
 $F \cdot l$, Theorem \ref{thm:Sym-Bob}.

 Since each class of integrable surfaces is defined
 by a real form of $\Lambda \mathfrak{sl}(2, \mathbb C)_\sigma$, there exists a
 unique semi-linear involution $\rho$ corresponding to each class of
 integrable surfaces. Then these semi-linear involutions 
 naturally define the pairs of semi-linear involutions on pairs of
 holomorphic potentials $\check \eta = (\eta, \tau)$.
 It follows that the generalized Weierstra{\ss} type representation for
 each class of integrable surfaces can be formulated by the above
 construction with a pair of holomorphic potentials which is invariant
 under a pair of semi-linear involutions, Theorem
 \ref{thm:DPWforint}.
 In this way we give a unified theory for all integrable surfaces.

 More precisely, in Section \ref{sc:Pre}, we give a brief review of the
 basic results for complex {\sc CMC} and {\sc CGC}-immersions. In
 Section \ref{subsc:holonull}, holomorphic null
 immersions and the basic facts for holomorphic null
 immersions are considered. In Section \ref{subsc:complexCMC}, the basic
 facts and results for complex {\sc CMC}-immersions are 
 given. Analogously to the complex {CMC}-immersions, 
 the definition and the basic facts for complex {\sc CGC}-immersions 
 are given, Theorem \ref{thm:Sym-Bob}.
 In Section \ref{subsc:realform}, the classification of real forms of $\Lambda
 \mathfrak{sl}(2, \mathbb C)_{\sigma}$ is given, Theorem
 \ref{thm:almostcompact} and Theorem \ref{thm:almostsplit}.

 In Section \ref{sc:Realforms}, we give a classification of integrable
 surfaces. In Section \ref{subsc:Integrablesurf}, it is shown that all
 integrable surfaces are obtained from real forms of 
 $\Lambda \mathfrak{sl}(2, \mathbb C)_\sigma$, 
 Theorem \ref{thm:compactsplit}.  In Section
 \ref{subsc:gaussmap}, the Gau{\ss} maps of integrable surfaces are 
 characterized using the Gau{\ss} map of a complex {\sc CMC}-immersion and
 the real forms of the complex extended framing, Theorem \ref{thm:Gaussmap}.

 In Section \ref{sc:DPW}, we give a construction of all integrable surfaces
 via the generalized Weierstra{\ss} type representation.
 In Section \ref{subsc:DPW}, the pairs of semi-linear
 involutions, which are determined from classes of integrable
 surfaces, are considered. Then a construction of all
 integrable surfaces is discussed via the pairs of holomorphic
 potentials which are invariant under the pairs of
 semi-linear involutions, Theorem \ref{thm:DPWforint}.
 
 In Appendix \ref{BasicResult}, we give the basic results for affine
 Kac-Moody Lie algebras and the loop algebras. In Section
 \ref{subsc:Affine}, Kac-Moody Lie algebras are considered.
 In Section \ref{loopgroups}, loop groups and a realization of affine
 Kac-Moody Lie algebras  via  the twice central extensions of loop
 algebras are discussed, Theorem \ref{thm:Kac}.
 In Section \ref{nsc:doubleIwasawa}, double
 loop groups are defined and the Iwasawa decomposition theorem 
 for the double loop groups are given, Theorem \ref{doublesplitting}.

\section{Preliminaries}\label{sc:Pre}
 In this preliminary section,  we give a brief review of the basic
 results for holomorphic null immersions, complex {\sc
 CMC}-immersions and complex {\sc CGC}-immersions. 
 We also give a brief review of the basic facts about loop 
 algebras and their real forms.
 
 Throughout this paper, $\mathbb C^3$ is identified with
 $\mathfrak{sl}(2, \mathbb C)$ as follows:
\begin{equation}\label{eq:ident1}
 (a, b, c)^t \in \mathbb C^3 \leftrightarrow -\frac{i a}{2} \sigma_1
  -\frac{i b}{2} \sigma_2 - \frac{i c}{2} \sigma_3 \in \mathfrak{sl} (2,
  \mathbb C)\;\;,
\end{equation}
where $\sigma_j\;(j=1, 2, 3)$ are Pauli matrices as follows:
\begin{equation}\label{eq:ident2}
\sigma_1 = \begin{pmatrix}0 & 1 \\ 1 & 0 \end{pmatrix}, \;\;\sigma_2 =
\begin{pmatrix} 0 & -i  \\ i & 0
\end{pmatrix}\;\; \mbox{and}\;\;\sigma_3 = \begin{pmatrix} 1 & 0 \\ 0 & -1\end{pmatrix}\;.
\end{equation}

 \subsection{Holomorphic null immersions in $\mathbb
  C^3$}\label{subsc:holonull}
 In this subsection, we show the basic results for holomorphic
 immersions $\varPsi$ from $\mathfrak D^2 \subset \mathbb C^2$ into
 $\mathbb C^3$. We give natural definitions of complex mean curvature (Definition
 \ref{def:Mean}) and complex Gau{\ss} curvature (Definition
 \ref{def:Gauss}) for a holomorphic immersion  analogous to the mean
 curvature and the Gau{\ss} curvature of a surface in $\mathbb R^3$.
 We refer to \cite{DKP:Complex} for more details.

 Let $\mathcal M$ be a simply connected $2$-dimensional Stein manifold, and let $\varPsi:
 \mathcal M \to \mathfrak{sl}(2, \mathbb C)$ be a holomorphic immersion,
 i.e., the complex rank of $d\varPsi$ is two. We consider the following bilinear
 form on $\mathfrak{sl}(2, \mathbb C) \cong \mathbb C^3$:
 \begin{equation}\label{eq:bilinear}
 \langle a, b \rangle = -2 {\rm Tr}\; a b\;,
 \end{equation}
 where $a, b \in \mathfrak{sl}(2, \mathbb C)$. We note that the bilinear
 form \eqref{eq:bilinear} is a $\mathbb C$-bilinear form on $\mathbb
 C^3$ by the identification \eqref{eq:ident1}. Then it is known that,
 for a neighborhood $\widetilde {\mathcal M}_p \subset \mathcal M$ around
 each point $p \in \mathcal {M}$, the bilinear form \eqref{eq:bilinear}
 induces a holomorphic Riemannian metric on $\widetilde{\mathcal M}_p$, 
 i.e., a holomorphic covariant symmetric 2-tensor
 $g$, see \cite{LeBrun:Complex-Riemannian} and \cite{DKP:Complex}. From
 \cite{DKP:Complex}, it is also known that there exist
 special coordinates $(z, w) \in \mathfrak D^2 \subset \mathbb C^2$ such
 that a holomorphic Riemannian metric $g$ can be
 written as follows:
 \begin{equation}\label{eq:nullmetric}
 g = e^{u(z, w)}dz dw\;,
 \end{equation}
 where $u(z, w) : \mathfrak D^2 \to \mathbb C$ is some holomorphic
 function. The special coordinates defined above are called {\it
 null coordinates}. From now on, we always assume a holomorphic immersion
 $\varPsi : \mathcal M \to \mathfrak{sl}(2, \mathbb C)$ has null
 coordinates. A holomorphic immersion with null
 coordinates is also called the {\it holomorphic null immersion}.

 \begin{Remark}
 The assumption ``Stein'' is used for the existence of the form of a
  holomorphic Riemannian metric $g =e^{u(z, w)} dz dw$ defined in
  \eqref{eq:nullmetric} and the existence of a well-defined pair of
  holomorphic  potentials on $\mathcal M$ for the generalized Weierstra{\ss}
  type representation  in Section \ref{sc:DPW}, see \cite{DKP:Complex} for more
  details.
 \end{Remark}
 We now define a vector $N \in \mathfrak{sl}(2, \mathbb C)$ as follows:
 \begin{equation}
 \label{eq:GaussMap}
 N := 2 i e^{-u} [\varPsi_{w},\; \varPsi_z ]\;,
 \end{equation}
 where the subscripts $z$ and $w$ denote the partial derivatives with
 respect to $z$ and $w$, respectively.
 It is easy to verify that  $\langle \varPsi_z, N\rangle =\langle
 \varPsi_w, N\rangle = 0$ and the $\langle N, N\rangle =1$. Thus $N$ is 
  a transversal vector to $ d \varPsi$. Therefore it is natural to call $N$ 
 the {\it complex Gau{\ss} map} of $\varPsi$.

 From \cite{DKP:Complex}, we quote the following theorem:
 \begin{Theorem}[\cite{DKP:Complex}]\label{thm:MovingFrame}
 Let $\varPsi : \mathcal M \to \mathbb C^3 (\cong \mathfrak{sl}(2, \mathbb C))$ be a holomorphic null
  immersion. Then there exists a $SL(2, \mathbb C)$ matrix $F$ such that 
  the following equations hold:
 \begin{equation}
 \label{eq:Lax-complex}
 \begin{array}{lcr}
 F_{z} = F U, \\
 F_{w} = F V,
 \end{array}
 \end{equation}
 where
\begin{equation}\label{eq:U-V}
\left\{
\begin{array}{lcr}
U =
\begin{pmatrix}
\frac{1}{4} u_z & -\frac{1}{2} H e^{u/2} \\
Q e^{-u/2} & -\frac{1}{4} u_{z}
\end{pmatrix},
 \\ \\
V =
\begin{pmatrix}
-\frac{1}{4} u_w & - R e^{-u/2} \\
\frac{1}{2} H e^{u/2} & \frac{1}{4} u_w
\end{pmatrix},
\end{array}
\right.
\end{equation}

 with $ Q := \langle \varPsi_{zz}, N \rangle $,  $R:=\langle
 \varPsi_{ww}, N\rangle $ and $H := 2 e^{-u} \langle \varPsi_{z w}, N
 \rangle $.  
 \end{Theorem}
  We call $F : \mathcal M \to SL(2, \mathbb C)$ the {\it moving
 frame} of $\varPsi$. Then the compatibility condition for the equations
 in \eqref{eq:Lax-complex} is
 \begin{equation}\label{eq:compatibilit}
 U_w - V_z + [V, U] = 0 .
 \end{equation}
 A direct computation shows that the equation \eqref{eq:compatibilit} can be rephrased as follows:
 \begin{equation}\label{eq:GC}
 \left\{
 \begin{array}{lcr}
 u_{zw} - 2 R Q e^{-u} + \frac{1}{2}H^2 e^u =0,\\[0.2cm]
 Q_w - \frac{1}{2} H_z e^u= 0, \\[0.2cm]
 R_z - \frac{1}{2} H_w e^u = 0.
 \end{array}
\right.
 \end{equation}
 The first equation in \eqref{eq:GC} will be called the {\it 
 complex Gau{\ss} equation}, and the second and third equations in
 \eqref{eq:GC} will be called the {\it complex Codazzi
 equations}. 
 
 From the discussion above we know that all holomorphic null immersions $\varPsi :
 \mathcal M \rightarrow \mathbb C^3$ satisfy the complex Gau{\ss}-Codazzi
 equations \eqref{eq:GC}. We note that, setting $\alpha = F^{-1} d
 F = U dz + V dw$, the equations in \eqref{eq:GC} are equivalent to 
 $$d \alpha + \frac{1}{2} [\alpha \wedge \alpha] = 0.$$

 Using the functions $u$, $Q$, $R$ and $H$ defined in
 \eqref{eq:nullmetric} and \eqref{eq:U-V} respectively, the symmetric quadratic form
 $I\!I : = - \langle d \varPsi , d N \rangle$ can be represented as follows:
 \begin{equation}\label{eq:second-fund}
 I\!I :=- \langle d \varPsi , d N\rangle = Q dz^2 + e^u H dz dw + R dw^2\;.
 \end{equation}
 The symmetric quadratic form $I \! I$ is called the {\it second
 fundamental form} for a holomorphic null immersion $\varPsi$.
 Then the complex mean curvature and the complex Gau{\ss} curvature for a
 holomorphic null immersion $\varPsi$ are defined as follows.
\begin{Definition}\label{def:Mean}
 Let $\varPsi : \mathcal M \to \mathbb C^3$ be a holomorphic null
 immersion. Then the function $H = 2 e^{-u}\langle \varPsi_{z
 w}, N \rangle$ will be called the {\rm complex mean curvature} of $\varPsi$.
\end{Definition}

\begin{Remark}\label{Rem:Tr}
 From the forms of the holomorphic metric $g$ defined in \eqref{eq:nullmetric}
 and the second fundamental form $I\!I$ defined in
 \eqref{eq:second-fund}, the equality $H = \tfrac{1}{2}{\rm Tr} (\tilde
 I^{-1} \cdot \widetilde {I\!I})$ holds, where
 $\tilde I$ (resp. $\widetilde {I\!I}$) is the coefficient matrix of $g$ (resp. $I\!I$).
\end{Remark}

\begin{Definition}\label{def:Gauss}
 We retain the notation in Remark \ref{Rem:Tr}. 
 Then the function $K = {\rm det} (\tilde I^{-1}  \cdot \widetilde
 {I\!I}) = H^2 - 4 e^{-2u} Q R$ will be called the {\rm complex Gau{\ss}
 curvature} of $\varPsi$.
\end{Definition}

\subsection{Complex CMC and CGC immersions in $\mathbb
  C^3$}\label{subsc:complexCMC}
 In this subsection, we give characterizations of complex constant
 mean curvature immersions via loop groups, see Appendix \ref{loopgroups} 
 for the definitions of loop groups.
 There is a useful formula representing complex {\sc CMC}-immersions, 
 which is a generalization of the Sym formula
 for {\sc CMC}-immersions in $\mathbb R^3$, see also
 \cite{DK:cyl}. There is also a formula for complex {\sc
 CGC}-immersions given by the parallel holomorphic immersions of complex
 {\sc CMC}-immersions with $H \in \mathbb C^*$.

 The notions of a complex  {\sc CMC}-immersion and a {\sc CGC}-immersion are defined
 analogous to the notions of a {\sc CMC}-immersion and a {\sc
 CGC}-immersion in $\mathbb R^3$, see also \cite{DKP:Complex}.
\begin{Definition}
 Let $\varPsi: \mathcal M \to \mathbb C^3$ be a holomorphic null
 immersion, and let $H$ (resp. K) be its complex mean curvature (resp. Gau{\ss} curvature).
 Then $\varPsi$ is called a {\it complex constant mean curvature ({\sc CMC} for short)
 immersion}  (resp. a {\it complex constant Gau{\ss} curvature ({\sc
 CGC} for short) immersion}) if $H$ (resp. $K$) is a complex constant.
\end{Definition}

\begin{Remark}
 Since we are interested in complexifications of {\sc CMC}
 (resp. {\sc CGC}) surfaces with nonzero mean curvature $H \in \mathbb R^*$
 (resp. Gau{\ss} curvature $K \in \mathbb R^*$), from now on, we always assume that the
 complex mean curvature $H$ (resp. the complex Gau{\ss} curvature $K$) is 
 a nonzero constant. 
\end{Remark}

 From \cite{DKP:Complex}, we quote the following characterizations
 of a complex {\sc CMC}-immersion:
\begin{Lemma}
 \label{lem:complex-CMC}
 Let $\mathcal M$ be a connected 2-dimensional Stein manifold,
 and let $\varPsi :\mathcal M \to \mathbb C^3 \cong \mathfrak{sl}(2, \mathbb C)$
 be a holomorphic null immersion. Further, let $Q$, $R$, $H$ and $N$ be the complex
 functions defined in \eqref{eq:U-V} and the {Gau\ss} 
 map defined in \eqref{eq:GaussMap}, respectively.
 Then the following statements are equivalent:
 \begin{enumerate}
 \item $H$ is a nonzero constant;
 \item $Q$ depends only on $z$ and $R$ depends only on
 $w$;
 \item $N_{z w} = \rho N$, for some holomorphic function $\rho :
       \mathcal M\rightarrow \mathbb C$.
 \item\label{itm:extend} There exists $\tilde F(z, w, \lambda) \in
       \Lambda SL(2, \mathbb C)_{\sigma}$ such that 
\begin{equation*}
\begin{array}{l}
       \tilde F(z, w, \lambda)^{-1} d\tilde F(z, w, \lambda)
	=\tilde U dz + \tilde V dw, \\

\end{array}
\end{equation*}
       where 
       \begin{equation*}
      \left\{
\begin{array}{lcr}
\tilde U =
\begin{pmatrix}
\frac{1}{4} u_z & -\frac{1}{2}\lambda^{-1} H e^{u/2} \\
\lambda^{-1} Q e^{-u/2} & -\frac{1}{4} u_{z}
\end{pmatrix},
 \\ \\
\tilde V =
\begin{pmatrix}
-\frac{1}{4} u_w & - \lambda R e^{-u/2} \\
\frac{1}{2}\lambda H e^{u/2} & \frac{1}{4} u_w
\end{pmatrix},
\end{array}
\right.
\end{equation*}

 \end{enumerate}
 and $\tilde F(z, w, \lambda =1)=F(z, w)$ is the moving frame of $\varPsi$ in \eqref{eq:Lax-complex}.
\end{Lemma}

 The $\tilde F(z, w, \lambda)$ defined in \eqref{itm:extend} of
 Lemma \ref{lem:complex-CMC} is called the {\it complex extended framing} of a complex
 {\sc CMC}-immersion $\varPsi$.
 From now on, for simplicity, the symbol $F(z, w, \lambda)$
 (resp. $U(z, w, \lambda)$ or $V(z, w, \lambda)$) is used instead of $\tilde
 F(z, w, \lambda)$ (resp. $\tilde U(z, w, \lambda)$ or $\tilde V(z, w,
\lambda)$).

 There is an immersion formula for a complex {\sc CMC}-immersion using
 the complex extended framing $F(z, w, \lambda)$ for a complex {\sc
 CMC}-immersion $\varPsi$, the so-called ``Sym formula'', see
 \cite{DKP:Complex}. We show a similar immersion formula for a
 complex {\sc CGC}-immersion using the same complex extended framing
 $F(z, w, \lambda)$ of a complex {\sc CMC}-immersion $\varPsi$.
 \begin{Theorem}
 \label{thm:Sym-Bob}
  Let $F(z, w, \lambda)$ be the complex extended framing of some complex
  CMC-immersion defined as in
  Lemma \ref{lem:complex-CMC}, and let $H$ be its nonzero complex constant mean curvature. We set
 \begin{equation}\label{eq:Sym-Bobenko}
\left\{
\begin{array}{l}
 \displaystyle \varPsi = -\frac{1}{2H}\left( i \lambda \partial_{\lambda} F (z, w, \lambda)
  \cdot F(z, w, \lambda)^{-1} + \frac{i}{2}F(z, w, \lambda) \sigma_3 F(z, w, \lambda)^{-1}\right), \\[0.1cm]
 \displaystyle  \varPhi = -\frac{1}{2H}\left( i \lambda \partial_{\lambda} F (z, w, \lambda)
  \cdot F(z, w, \lambda)^{-1} \right),
 \end{array}
\right.
 \end{equation}
 where $\sigma_3$ has been defined in \eqref{eq:ident2}.
 Then $\varPsi$ (resp. $\varPhi$) is, for
  every $\lambda \in \mathbb C^\ast$, a complex constant mean curvature
  immersion (resp. complex constant Gau{\ss}ian curvature
  immersion, possibly degenerate) in
  $\mathbb C^3$ with complex mean curvature $H \in  \mathbb C^*$
  (resp. complex Gau{\ss}
  curvature $K = 4 H^2 \in \mathbb C^*$),
  and the Gau{\ss} map of  $\varPsi$ (resp. $\varPhi$) can be described
  by $\tfrac{i}{2}F(z, w, \lambda)\sigma_3 F(z, w, \lambda)^{-1}$.
 \end{Theorem}

\begin{proof}
 The proof for complex {\sc CMC}-immersions follows from
 \cite{DKP:Complex}.  We show that the second formula $\varPhi$ in
 \eqref{eq:Sym-Bobenko} defines a complex {\sc CGC}-immersion.
 Let $\varPhi$ be a map in the second formula in \eqref{eq:Sym-Bobenko}.
 Let $N=\tfrac{i}{2}F(z, w, \lambda)\sigma_3 F(z, w, \lambda)^{-1}$ be
 the complex Gau{\ss} map for $\varPsi$. Since $\langle N, N \rangle =1$
 and the relation $\varPhi = \varPsi + \tfrac{1}{2 H} N$ 
 holds for the formulas in \eqref{eq:Sym-Bobenko}, $N$ is also the Gau{\ss} map for
 $\varPhi$, i.e., $\langle \varPhi_{z}, N\rangle$ =
 $\langle\varPhi_{w}, N \rangle = 0$.  We also denote the Gau{\ss} map for
 $\varPhi$ by $N$. 

 We then compute the holomorphic metric $g$ and the second fundamental
 form $I\!I$ for the holomorphic map $\varPhi$. Using the bilinear form
 in \eqref{eq:bilinear}, we have
\begin{equation*}
\left\{
 \begin{array}{l}
  \displaystyle \langle \varPhi_{z}, \varPhi_{z}\rangle = -2 {\rm Tr} \left(\varPhi_{z}
  \cdot \varPhi_{z} \right) = \frac{\lambda^2}{2 H^2} {\rm Tr}\left({\rm
  Ad} (F)U_{\lambda}^2\right), \\
 \displaystyle \langle \varPhi_{w}, \varPhi_{w} \rangle = -2 {\rm Tr} \left(\varPhi_{w}
  \cdot \varPhi_{w} \right) =\frac{\lambda^2}{2 H^2} {\rm Tr}
 \left({\rm  Ad} (F)V_{\lambda}^2\right),
\end{array}
\right.
 \end{equation*}
 where $U$ and $V$ are defined in Lemma \ref{lem:complex-CMC}, and the
 subscript $z$ (resp. $w$ or $\lambda$) denotes the partial derivative
 with respect to $z$ (resp. $w$ or $\lambda$). Since the trace of a matrix
 is invariant under the map ${\rm Ad} (F)$, and using the form of $U$ in
 Lemma \ref{lem:complex-CMC}, we  have
\begin{equation*}
\left\{
 \begin{array}{l}
 \displaystyle \langle \varPhi_{z}, \varPhi_{z}\rangle = \frac{\lambda^2}{2 H^2} {\rm Tr}
  (U_{\lambda}^2) = -\frac{1}{2 H^2}
 \lambda^{-2} HQ, \\[0.3cm]

\displaystyle \langle \varPhi_{w}, \varPhi_{w}\rangle  = \frac{\lambda^2}{2 H^2} {\rm Tr}
  (V_{\lambda}^2) =-\frac{1}{ 2H^2}
\lambda^2 H R.
\end{array}
\right.
 \end{equation*}
 Using again the invariace of trace of a matrix under the map ${\rm Ad (F)}$ and
 the forms of $U$ and $V$ in Lemma \ref{lem:complex-CMC}, we have
\begin{flalign*}
 \langle \varPhi_{z}, \varPhi_{w}\rangle &= -2 {\rm Tr}\left( -\frac{i \lambda}{2 H} {\rm
  Ad} (F)U_{\lambda} \times-\frac{i\lambda}{2 H} {\rm
  Ad} (F)V_{\lambda} \right)\\  &= \frac{1}{2 H^2} \left(\frac{1}{4} H^2 e^u +
 Q R e^{-u}\right).
\end{flalign*}
 Therefore, we have the following first fundamental form for the
 holomorphic map $\varPhi$:
 \begin{equation}\label{eq:firstfork}
 g = \begin{pmatrix}dz &  dw\end{pmatrix}
     \begin{pmatrix} 
     -\frac{1}{2 H^2} \lambda^{-2} H Q & \frac{1}{2 H^2}\left( \frac{1}{4} H^2 e^u + Q R e^{-u}\right) \\[0.5cm]
     \frac{1}{2 H^2}\left( \frac{1}{4} H^2 e^u + Q R e^{-u}\right) & -\frac{1}{2 H^2} \lambda^{2} H R 
    \end{pmatrix} 
 \begin{pmatrix} dz \\ dw\end{pmatrix}
\;.
\end{equation}
 Let $\tilde I$ denote the coefficient matrix for $g$ in
 \eqref{eq:firstfork}. Then it is easy to verify that $\det \tilde I =
 -(\frac{1}{4} H^2 e^u - Q R e^{-u})^2/(4 H^4)$. 
 Thus the holomorphic map $\varPhi$  actually defines a holomorphic immersion
 under the condition $e^{2 u} \neq 4 H^{-2} Q R$.

 Next, the second fundamental form $I\!I$ for $\varPhi$ is computed
 as follows. Using again the invariace of the trace of a matrix under the map ${\rm Ad (F)}$, 
 $\langle \varPhi_{z}, N_{z}\rangle$ is computed as follows:
 \begin{flalign*}
 \langle \varPhi_{z}, N_{z}\rangle  & = -2 {\rm Tr}\left( -\frac{i
  \lambda}{2 H} {\rm Ad} (F)U_{\lambda}
  \times \frac{i}{2} {\rm Ad} (F) [U,\sigma_3]\right)
  \\ & = -\frac{1}{2 H} {\rm Tr}
  (U_{\lambda}\cdot [U, \sigma_3])\;.
\end{flalign*}
 From the form of $U$ in Lemma
 \ref{lem:complex-CMC}, we have $\langle \varPhi_{z}, N_{z}\rangle = 0$.
 A similar argument holds for $\langle \varPhi_{w}, N_{w} \rangle$, where $U$ is replaced by $V$.
 Thus we have  $\langle \varPhi_{w}, N_{w} \rangle = 0$.  
 Using again the invariance of the trace of a matrix under the map ${\rm Ad} (F)$
 and the form of $U$ in Lemma
 \ref{lem:complex-CMC}, we obtain
 \begin{flalign*}
 \langle \varPhi_{z}, N_{w}\rangle &= -2 {\rm Tr}\left( -\frac{i
  \lambda}{2 H} {\rm Ad} (F)U_{\lambda}
  \times \frac{i}{2} {\rm Ad} (F) [V,\sigma_3]\right)
  \\& = -\frac{1}{2 H} \left(\frac{1}{2}
  H^2 e^u - 2 QR e^{-u}\right).
\end{flalign*}

 Since $N$ is the Gau{\ss} map of $\varPhi$, i.e., $\langle \varPhi_{z}, N\rangle$ =
 $\langle\varPhi_{w}, N \rangle = 0$ and $\langle N,
 N\rangle=1$, we obtain $\langle \varPhi_{z}, 
 N_{w}\rangle =- \langle \varPhi_{wz}, N\rangle = \langle
 \varPhi_{w}, N_{z}\rangle$.

 Finally, the second fundamental form $I\!I$ for $\varPhi$ has 
 the following form:
 \begin{equation}\label{eq:secondfork}
 I\!I = \begin{pmatrix}dz &  dw\end{pmatrix} 
        \begin{pmatrix} 0 &  \frac{1}{2 H}\left(     \frac{1}{2} H^2 e^u -2 Q R e^{-u}\right) \\[0.5cm]
                      \frac{1}{2 H}\left( \frac{1}{2} H^2 e^u -2 Q R e^{-u}\right)& 0 
        \end{pmatrix}
        \begin{pmatrix} dz \\ dw\end{pmatrix}
       \;.
 \end{equation}
 Let us denote the coefficient matrix of $I\!I$ by $\widetilde {I\!I}$. 
 Then, using  \eqref{eq:firstfork} and \eqref{eq:secondfork}, the
 complex Gau{\ss} curvature $K$ for $\varPhi$ is computed as
 \begin{equation*}
 K = {\rm det }( \tilde I^{-1} \cdot \widetilde {I\!I} ) = 4 H^2 \in
  \mathbb C^*\;.
 \end{equation*}
  This completes the proof.
\end{proof}
 Since the Gau{\ss} maps of a complex {\sc CMC}-immersion $\varPsi$
 and the corresponding complex {\sc CGC}-immersion $\varPhi$ are the same, which is
 $N = \tfrac{i}{2}F(z, w, \lambda)\sigma_3 F(z, w, \lambda)^{-1}$,
 we have the following corollary:
\begin{Corollary}
 Let $\varPsi$ be a complex {\sc CMC}-immersion, and let $H$ (resp. $u$,
 $Q$ and $R$) be its nonzero constant mean curvature (resp. the
 functions defined in \eqref{eq:nullmetric} and Theorem \ref{thm:MovingFrame}). 
 Moreover, let us assume $e^{2 u} \neq 4 H^{-2} Q R$. Then there exists the parallel complex
 {\sc CGC}-immersion with Gau{\ss} curvature $K = 4 H^2 \in \mathbb C^*$.
\end{Corollary}

\subsection{Real forms of $\Lambda \mathfrak{sl}(2, \mathbb
  C)_{\sigma}$}\label{subsc:realform}
 In this subsection, we give the classification of real forms
 for the twisted $\mathfrak{sl}(2, \mathbb C)$ loop algebra $\Lambda
 \mathfrak{sl}(2, \mathbb C)_{\sigma}$, see Appendix \ref{BasicResult}
 for the notation and the definitions of loop algebras.

 First we recall the basic facts about real forms for complex Kac-Moody 
 Lie algebras.
 Let $\mathfrak g$ be a Kac-Moody Lie algebra over $\mathbb C$. Then a Lie
 subalgebra $\mathfrak g_{\mathbb R} \subset \mathfrak g$ over
 $\mathbb R$ will be called the {\it real form} of $\mathfrak g$ if there is an
 isomorphism between $\mathfrak g$ and the complexification $\mathfrak
 g_{\mathbb R} \otimes \mathbb C$.
 We note that all real forms of $\mathfrak g$ correspond to semi-linear
 involutions of $\mathfrak g$, see, for example
 \cite{BR:Kac-Moody},  i.e., each real form is
 defined by an automorphism $\rho$ of $\mathfrak g$ such that
 \begin{equation}\label{eq:semi-involution}
\left\{
\begin{array}{l}
 \rho^2 = {\rm id}, \\
 \rho(\ell x) = \bar \ell \rho (x)\;\;\; \mbox{for}\;\; \ell \in
  \mathbb C.
 \end{array}
\right.
 \end{equation}

 Let $\mathfrak h$ denote the {\it standard Cartan subalgebra} of $\mathfrak
 g$, which is a maximal $\mbox{ad}(\mathfrak g)$--diagonalizable
 subalgebra of $\mathfrak g$. Let $\Delta$ be the corresponding root
 system, and let $\mathfrak g_{\alpha}$ denote the root space corresponding
 to $\alpha$ in $\Delta$. Then the root space
 decomposition for $\mathfrak g$ is as follows:
 \begin{equation*}
 \mathfrak g = \mathfrak h \oplus \left(\bigoplus_{\alpha \in \Delta}
				   \mathfrak g_{\alpha}\right)
 \;.
 \end{equation*}
 It is known that $\Delta$ can be decomposed as $\Delta = \Delta^{+}
 \cup \Delta^{-}$, where $\Delta^{+}$ (resp. $\Delta^{-}$) is the set of
 positive (resp. negative) roots.  

 Then a subalgebra of $\mathfrak g$ is said to be a {\it Borel subalgebra} if it is a maximal
 completely solvable subalgebra. And the {\it standard positive (resp. negative) Borel
 subalgebra} $\mathfrak b^{+}$ (resp. $\mathfrak b^{-}$) of $\mathfrak
 g$ is defined as follows:
 \begin{equation*}
 \mathfrak b^{\pm} = \mathfrak h \oplus \left(\bigoplus_{\beta \in \Delta^{\pm}}
				   \mathfrak g_{\beta}\right)
 \;.
 \end{equation*}
 If a linear or semi-linear automorphism $\rho$ for $\mathfrak g$ transforms a
 Borel subalgebra into a Borel subalgebra of the same (resp. opposite)
 sign, then $\rho$ is said to be the {\it first kind} (resp. {\it second kind}).

 We now give definitions of the almost split real forms and the almost
 compact real forms of $\mathfrak g$.
\begin{Definition}
 Let $\mathfrak g$ be a Kac-Moody Lie algebra over $\mathbb C$, and let
 $\mathfrak g_{\mathbb R}$ be a real form of $\mathfrak g$. Moreover,
 let  $\rho$ be the  semi-linear involution corresponding to $\mathfrak
 g_{\mathbb R}$.  Then the real
 form $\mathfrak g_{\mathbb R}$ is called {\rm almost split}
 (resp. {\rm almost compact}) if the corresponding semi-linear
 involution $\rho$ is of the first kind (resp. second kind).
\end{Definition}

 It is clear that the real subalgebra of $\mathfrak g$ generated by 
 $\{\mathfrak h, e_j, f_j \;;\;j = 1, 2, \dots, n \}$, the Cartan
 subalgebra and the Chevalley generators of
 the Kac-Moody Lie algebra $\mathfrak g$, see Appendix
 \ref{subsc:Affine}, is
 an almost split real form, which is called the {\it standard split
 form}. The corresponding semi-linear involution of 
 the first kind $\sigma_n^{\prime}$ is called the {\it standard normal
 semi-involution} of $\mathfrak g$. The map $e_j \mapsto -f_j$, $f_j
 \mapsto -e_j$ and $h \mapsto -h$, $h \in \mathfrak h$ for  $\{\mathfrak h, e_j,
 f_j \;;\;j = 1, 2, \dots, n \}$ can be extended to an involution
 $\omega$ of $\mathfrak g$. The $\omega$ is called the {\it Cartan
 involution} of $\mathfrak g$. It is known that the standard normal
 semi-involution and the Cartan involution commute.
\begin{Definition}
 Let $\omega$ and $\sigma_n^{\prime}$ be the Cartan involution and the
 standard normal semi-involution respectively, and let
 $\omega^{\prime}$ be $\omega^{\prime} = \sigma_n^{\prime} \omega =
 \omega \sigma_n^{\prime}$. Then $\omega^{\prime}$ is called {\rm the
 standard Cartan semi-involution},
 and the corresponding almost compact real form is called {\rm the
 standard compact form}. Moreover, a conjugation of $\omega^{\prime}$ is
 called a {\rm Cartan semi-involution}.
\end{Definition}
 We quote the following theorem about the real forms of
 Kac-Moody Lie algebras \cite{BBBR:almostsplit},
 \cite{BR:almostcompact}.
\begin{Theorem}[Theorem 4.4 in \cite{BBBR:almostsplit}, Proposition 2.9
 in \cite{BR:almostcompact}]\label{thm:Ross}
 Let us consider the following:
\begin{enumerate}
\item[(1)] The semi-linear involutions $\rho$ of $\mathfrak{g}$ of the
	   second kind (resp. the first kind).
\item[(2)] The involutions $\theta$ of $\mathfrak{g}$ of the first
	   kind (resp. the second kind).
\item[(3)] The relation $\rho \thickapprox \theta$ if and only if
 \begin{itemize}
 \item[(a)] $\omega^{\prime} = \theta \rho = \rho \theta$ is a Cartan semi-involution.
 \item[(b)] $\theta$ and $\rho$ stabilize the same Cartan subalgebra $\mathfrak{h}$.
 \item[(c)] $\mathfrak{h}$ is contained in a minimal $\rho$-stable
	    positive parabolic subalgebra.
 \end{itemize}
\end{enumerate}
 Then the relation induces a bijection between the conjugacy classes
 under $Aut(\mathfrak{g})$ of semi-linear involutions of the second kind
 (resp. the first kind) and conjugacy classes of involutions of
 the first kind (resp. the second kind).
\end{Theorem}
 We also quote the following theorem about the classification of the
 involutions of the affine Kac-Moody Lie algebra of $A_1^{(1)}$ type
 \cite{Kob:AutoKac}.
\begin{Theorem}[Theorem 3 in \cite{Kob:AutoKac}]\label{thm:ZKob}
 All involutions on the affine Kac-Moody Lie algebra of type $A_1^{(1)}$
 are given as follows:
 \begin{itemize}
\item[(a)] $e_1 \longmapsto - e_1, \;\;\;\; f_1 \longmapsto -f_1, \;\;\;\; e_2 \longmapsto
	     -e_2, \;\;\;\; f_2 \longmapsto - f_2$,\vspace{0.2cm}

\item[(a$^{\prime}$)]$e_1 \longmapsto e_2, \;\;\;\; f_1 \longmapsto f_2, \;\;\;\; e_2 \longmapsto
	     e_1, \;\;\;\; f_2 \longmapsto f_1$,\vspace{0.2cm}

\item[(b)]$e_1 \longmapsto e_1, \;\;\;\; f_1 \longmapsto f_1, \;\;\;\; e_2 \longmapsto
	     -\frac{1}{2}[[f_2, f_1],f_1], \;\;\;\; f_2 \longmapsto -
	     \frac{1}{2}[[e_2, e_1],e_1]$,\vspace{0.2cm}

\item[(b$^{\prime}$)]$e_1 \longmapsto -e_1, \;\;\;\; f_1 \longmapsto -f_1, \;\;\;\; e_2 \longmapsto
	     \frac{1}{2}[[f_2, f_1],f_1], \;\;\;\; f_2 \longmapsto 
	     \frac{1}{2}[[e_2, e_1],e_1]$,\vspace{0.2cm}

\item[(b$^{\prime \prime}$)]$e_1 \longmapsto f_2, \;\;\;\; f_1 \longmapsto e_2, \;\;\;\; e_2 \longmapsto
	     f_1, \;\;\;\; f_2 \longmapsto e_1$,\vspace{0.2cm}

\item[(c)]$e_1 \longmapsto e_1, \;\;\;\; f_1 \longmapsto f_1, \;\;\;\; e_2 \longmapsto
	     -e_2, \;\;\;\; f_2 \longmapsto - f_2$,
 \end{itemize}
 where $e_j, f_j$ for $j \in \{1, 2\}$  are the Chevalley generators of
 the affine Kac-Moody Lie algebra of $A_1^{(1)}$ type. Moreover the cases
 {\rm (a), (a$^{\prime}$)} and {\rm (c)} (resp.  {\rm
 (b)}, {\rm (b$^{\prime}$)} and {\rm (b$^{\prime \prime}$)}) 
 are involutions of the first kind (resp. the second kind). 
\end{Theorem}
 It is well known that the twice central
 extensions of the untwisted $\mathfrak{sl}(2, \mathbb C)$ loop algebra
 $\Lambda \mathfrak{sl}(2, \mathbb C)$ is an affine
 Kac-Moody Lie algebra of $A_{1}^{(1)}$ type, see Theorem
 \ref{thm:Kac}.  It is also known that the
 twisted  $\mathfrak{sl}(2, \mathbb C)$ loop algebra $\Lambda
 \mathfrak{sl}(2, \mathbb C)_{\sigma}$ and the
 untwisted $\mathfrak{sl}(2, \mathbb C)$ loop algebra $\Lambda \mathfrak{sl}(2, \mathbb C)$
 are isomorphic by the following map from $\Lambda \mathfrak{sl} (2,
 \mathbb C)$ to $\Lambda \mathfrak{sl} (2, \mathbb C)_\sigma$, see also \cite{Kac:Kac-Moody}:
 \begin{equation}\label{eq:isountwist}
 g(\lambda) \in \Lambda \mathfrak{sl} (2, \mathbb C) \mapsto {\rm Ad}
  \left(\begin{smallmatrix} \sqrt{\lambda} & 0 
  \\0& \sqrt{\lambda}^{-1}\end{smallmatrix}\right)  g(\lambda^2) \in \Lambda \mathfrak{sl} (2, \mathbb C)_{\sigma}\;.
 \end{equation}
 Therefore we have the following
 classification of all real forms for $\Lambda \mathfrak{sl}(2, \mathbb C)_\sigma$.
 \begin{Theorem}\label{thm:almostcompact}
 Let $\mathfrak{c}_j$ for $j \in \{1, 2, 3, 4 \}$ be the following
  involutions on $\Lambda \mathfrak{sl} (2, \mathbb C)_\sigma$:
 \begin{equation}\label{eq:classcom}
 \begin{array}{ll}
 \mathfrak{c}_1 : g(\lambda) \mapsto -\overline{g(-1/\bar
  \lambda)}^{t},\;\; & \mathfrak{c}_2: g(\lambda) \mapsto \overline{g \left(- 1/\bar \lambda\right)},\\
 
 \mathfrak{c}_3: g(\lambda) \mapsto  - \overline{g \left(1/\bar \lambda\right)}^t\;,\;\; &
  \mathfrak{c}_4: g(\lambda) \mapsto
   -{\rm Ad} \left(\begin{smallmatrix} 1/\sqrt{i} & 0 \\ 0 & \sqrt{i}
	    \end{smallmatrix}\right) \overline{g(i/ \bar \lambda)}^{t},
 \end{array}
\end{equation}
 where $g(\lambda) \in \Lambda \mathfrak{sl}(2, \mathbb C)_\sigma$.
 Then, the almost compact real forms of
  $\Lambda \mathfrak{sl}(2, \mathbb C)_{\sigma}$ 
 are the following real Lie subalgebras of $\Lambda \mathfrak{sl}(2, \mathbb  C)_{\sigma}$:
\begin{equation}
 \Lambda \mathfrak{sl}(2, \mathbb C)_{\sigma}^{(\mathfrak{c}, j)} = \left\{ g(\lambda ) \in \Lambda
 \mathfrak{sl}(2, \mathbb C)_{\sigma} \;\left|\; \mathfrak{c}_j \circ g (\lambda) = g(\lambda) \right. \right\}
\;\;\mbox{for}\;\;j\in \{1, 2, 3, 4\}.
\end{equation}
\end{Theorem}

\begin{Theorem}\label{thm:almostsplit}
 Let $\mathfrak{s}_j$ for $j \in \{1, 2, 3\}$ be the following
  involutions on $\Lambda \mathfrak{sl} (2, \mathbb C)_\sigma$:
 \begin{equation}\label{eq:classsplit}
 \begin{array}{ll}
 \mathfrak{s}_1 : g(\lambda) \mapsto -\overline{g(-\bar
  \lambda)}^{t},\;\; & 
 \mathfrak{s}_2 : g(\lambda) \mapsto \overline{g \left(- \bar \lambda\right)},\\
  \mathfrak{s}_3: g(\lambda) \mapsto  - {\rm Ad}
	 \left(\begin{smallmatrix} \lambda & 0 \\ 0 & \lambda^{-1}
	       \end{smallmatrix}\right) \overline{g \left( \bar
	       \lambda\right)}^t\;,\;\; &

 \end{array}
\end{equation}
 where $g(\lambda) \in \Lambda \mathfrak{sl}(2, \mathbb C)_\sigma$.
 Then, the almost split real forms of $\Lambda
 \mathfrak{sl}(2, \mathbb C)_{\sigma}$  are the following real Lie
 subalgebras of $\Lambda \mathfrak{sl}(2, \mathbb  C)_{\sigma}$:
\begin{equation}
 \Lambda \mathfrak{sl}(2, \mathbb C)_{\sigma}^{(\mathfrak{s}, j)} = \left\{ g(\lambda ) \in \Lambda
 \mathfrak{sl}(2, \mathbb C)_{\sigma} \;\left|\; \mathfrak{s}_j \circ g (\lambda) = g(\lambda) \right. \right\}
 \;\;\mbox{for}\;\;j\in \{1, 2, 3\}.
\end{equation}
\end{Theorem}
\begin{proof}
 Since all linear involutions of the first kind and the
 second kind for the affine Kac-Moody Lie algebra of $A_1^{(1)}$ type
 are classified in Theorem \ref{thm:ZKob}, and using Theorem \ref{thm:Ross}, 
 all semi-linear involutions $\rho$ of the first kind (resp. the second
 kind) can be represented as follows:
 $$
 \rho = \omega^{\prime} \theta\;,
 $$
 where $\omega^{\prime}$ is the Cartan semi-involution and
 $\theta$ is the linear involution of the second kind (resp. the
 first kind).  Noting that the identity map is also the trivial involution of
 the first kind, we have the four classes of semi-linear involutions of 
 the second kind and the three classes of semi-linear involutions of the
 first kind. Since the affine Kac-Moody Lie algebra of $A_1^{(1)}$ type can be 
 realized by the twice central extensions of the loop algebra $\Lambda
 \mathfrak{sl}(2, \mathbb C)$ (see Theorem \ref{thm:Kac}), the real
 forms of the loop algebra $\Lambda \mathfrak{sl}(2, \mathbb C)$ are
 derived. Finally, using the isomorphism
 \eqref{eq:isountwist}, the real forms of $\Lambda
 \mathfrak{sl}(2, \mathbb C)$ are transformed into the real forms of
 $\Lambda \mathfrak{sl} (2, \mathbb C)_{\sigma}$, which are obtained in
 \eqref{eq:classcom} and \eqref{eq:classsplit} by a direct calculation.
 This completes the proof. 
\end{proof}
\begin{Remark}
 The identification \eqref{eq:isountwist} implies, in the untwisted
 setting, the involution $\mathfrak{s}_3$ can be rephrased as follows:
 \begin{equation}\label{eq:invs3}
 \mathfrak{s}_3: g(\lambda) \mapsto -\overline{g \left( \bar
							  \lambda\right)}^t
\;\; \mbox{for}\;\;g(\lambda) \in \Lambda \mathfrak{sl}(2, \mathbb C).
 \end{equation}
 From now on
 we use the involution \eqref{eq:invs3} instead of the original
 involution $\mathfrak{s}_3$ in \eqref{eq:classsplit}.
\end{Remark}
%
\section{Real forms of complex CGC-immersions}\label{sc:Realforms}
 In this section, we show  one of the main theorems in this paper (Theorem
 \ref{thm:compactsplit}), which is the classification of ``integrable
 surfaces'' obtained from all the real forms of the twisted $\mathfrak{sl}(2,
 \mathbb C)$ loop algebra $\Lambda \mathfrak{sl}(2, \mathbb
 C)_{\sigma}$.

\subsection{Integrable surfaces as real forms of complex
  CGC-immersions}\label{subsc:Integrablesurf}
 Let $F(z, w, \lambda) \in \Lambda SL(2, \mathbb C)_{\sigma}$ be the 
 complex extended framing of some complex {\sc CGC}-immersion
 $\varPhi$. And let  $\alpha(z, w, \lambda) = F(z, w, \lambda)^{-1} d F(z, w,
 \lambda)$ be the Maurer-Cartan form of $F(z, w, \lambda)$. 
 From the forms of $U$ and $V$ defined as in Lemma \ref{lem:complex-CMC}, we
 set $\alpha_i \; (i \in \{-1, 0, 1\})$ as follows:
\begin{equation}\label{eq:alpha}
\alpha(z, w, \lambda) =  F^{-1} d F= U dz + V dw = \lambda^{-1} \alpha_{-1} + \alpha_0 + \lambda \alpha_1 \;, 
\end{equation}
 where 
\begin{equation}\label{eq:alpha2}
\left\{
\begin{array}{l}
\alpha_{-1} = \begin{pmatrix}0 & -\frac{1}{2} H e^{u/2}dz, \\ Q
        e^{-u/2}dz & 0\end{pmatrix},\\[0.5cm]
\alpha_{0} =\begin{pmatrix} \frac{1}{4} u_z dz - \frac{1}{4} u_w dw & 0
	    \\ 0 & -\frac{1}{4} u_z dz + \frac{1}{4} u_w dw\end{pmatrix},\\[0.5cm]
\alpha_{1} =\begin{pmatrix} 0 & -R e^{-u/2}dw \\ \frac{1}{2} H e^{u/2}dw
	    & 0 \end{pmatrix}.
\end{array}
\right.
\end{equation}
 We denote the space of $\Lambda \mathfrak{sl}(2, \mathbb C)_{\sigma}$ valued
 $1$-forms by $\Omega (\Lambda \mathfrak{sl}(2, \mathbb C)_\sigma)$.
 Similar to the involutions in Theorem \ref{thm:almostcompact}
 (resp. Theorem \ref{thm:almostsplit}), we define the involutions
 $\tilde {\mathfrak{c}}_j$ (resp. $\tilde{\mathfrak{s}}_j$) for $
 g(\lambda) \in \Omega
 (\Lambda \mathfrak{sl}(2, \mathbb C)_\sigma)$ as follows:
{\small
 \begin{equation}\label{eq:inv-1-forms}
\begin{array}{ll}
\left\{
 \begin{array}{l}
 \tilde{\mathfrak{c}}_1 : g(\lambda) \mapsto -\overline{g(-1/\bar
  \lambda)}^{t},\;\; \\[0.2cm] 
 \tilde{\mathfrak{c}}_2: g(\lambda) \mapsto \overline{g \left(- 1/\bar \lambda\right)},\\[0.2cm]
  \tilde{\mathfrak{c}}_3: g(\lambda) \mapsto  - \overline{g \left(1/\bar \lambda\right)}^t\;,\;\; \\[0.2cm]
 \tilde{\mathfrak{c}}_4: g(\lambda) \mapsto  -{\rm Ad} \left(\begin{smallmatrix} 1/\sqrt{i} & 0 \\ 0 & \sqrt{i}
	    \end{smallmatrix}\right) \overline{g(i/ \bar \lambda)}^{t},
 \end{array}
\right.
\hspace{1cm}
\left\{
 \begin{array}{l}
\tilde{ \mathfrak{s}}_1 : g(\lambda) \mapsto -\overline{g(-\bar
 \lambda)}^{t},\;\; \\[0.2cm]
 \tilde{ \mathfrak{s}}_2 : g(\lambda) \mapsto \overline{g \left(- \bar
							   \lambda\right)}, \\[0.2cm]
\tilde{ \mathfrak{s}}_3: g(\lambda) \mapsto  - \overline{g \left( \bar \lambda\right)}^t.
 \end{array}
\right.
\end{array}
\end{equation}
}

 Then the real forms of $\Omega(\Lambda \mathfrak{sl} (2, \mathbb
 C)_{\sigma}^{(\mathfrak{c}, j )})$ are defined as follows:
 \begin{equation}\label{eq:involutionscj}
 \begin{array}{l}
 \Omega(\Lambda \mathfrak{sl} (2, \mathbb C)_{\sigma}^{(\mathfrak{c}, j )}) = 
                            \left\{ g  \in \Omega(\Lambda
			     \mathfrak{sl}(2, \mathbb C)_{\sigma})
			     \;\left|\; \tilde {\mathfrak{c}}_j \circ g(\lambda ) = 
				g(\lambda) \right. \right\}\;, \\

 \Omega(\Lambda \mathfrak{sl} (2, \mathbb C)_{\sigma}^{(\mathfrak{s}, j )}) = 
                            \left\{ g \in \Omega(\Lambda
			     \mathfrak{sl}(2, \mathbb C)_{\sigma})
			     \;\left|\; \tilde{\mathfrak{s}}_j \circ g(\lambda ) = 
				g(\lambda) \right. \right\}\;.
\end{array}
\end{equation}
 From now on, for simplicity, we use the symbols $\mathfrak{c}_j$ and
 $\mathfrak{s}_j$ instead of $\tilde{\mathfrak{c}}_j$ and $\tilde{\mathfrak{s}}_j$ 
 respectively.
  We now consider the following conditions on $\alpha(z, w, \lambda)$:
\begin{itemize}
\item\label{itm:Almostcompact} {\bf Almost compact cases $(C, j)$:} $\alpha(z,
       w, \lambda)$ is an element in the real form $\Omega(\Lambda
       \mathfrak{sl} (2, \mathbb C)_{\sigma}^{(\mathfrak{c}, j )})$. \\[0.05cm]

\item {\bf Almost split cases $(S, j)$:} $\alpha(z, w, \lambda)$ 
        is an element in the real form $\Omega(\Lambda
       \mathfrak{sl} (2, \mathbb C)_{\sigma}^{(\mathfrak{s}, j )})$.
\end{itemize}
 A straightforward computation shows that the conditions above,
 which are the almost compact cases $(C, j)$ and the
 almost split cases $(S, j)$, are equivalent to the
 following equations for $\alpha_i \;(i \in\{ -1, 0, 1\})$:
\begin{equation}\label{eq:conditionA}
\left\{
\begin{array}{ll}
\alpha_0 = - \overline{ \alpha_0}\;\;\mbox{and} \;\;\alpha_{\pm j} = \overline
 {\alpha_{\pm 1}}^t\; & \mbox{for the $(C, 1)$ or $(S, 1)$ case,} \\
\alpha_0 = \overline{ \alpha_0}\;\;\mbox{and}\;\;\alpha_{\pm j} = - \overline
 {\alpha_{\pm 1}}\; & \mbox{for the $(C, 2)$ or $(S, 2)$ case,}  \\
\alpha_0 = - \overline{ \alpha_0}\;\;\mbox{and}\;\;\alpha_{\pm j} = - \overline
 {\alpha_{\pm 1}}^t\; & \mbox{for the $(C, 3)$ or $(S, 3)$ case,} \\[0.1cm]
\alpha_0 = -\overline{ \alpha_0}\;\;\mbox{and}\;\;\alpha_{-1} = i {\rm
 Ad } \left(\begin{smallmatrix}  1/\sqrt{i} & 0 \\ 0 & \sqrt{i}\end{smallmatrix}\right)\overline
 {\alpha_{1}}^t\; & \mbox{for the $(C, 4)$ case,}

\end{array}
\right.
\end{equation}
 where $j = -1$ (resp. $j = 1$) if $\alpha$ satisfies one of the
 conditions for almost compact cases (resp. almost split cases).
 From the symmetry between $\alpha_1$ and $\alpha_{-1}$ for the almost compact
 cases and the symmetries on each of $\alpha_1$ and $\alpha_{-1}$ for the almost
 split cases, we obtain 
\begin{equation}\label{eq:coordinates}
\left\{
\begin{array}{l}
w = \bar z\;\; \mbox{for the almost compact cases $(C, j)$,}  \\
z = \bar z\;\;\mbox{and}\;\; w = \bar w\;\;\mbox{for the almost split
 cases $(S, j)$.}
\end{array}
\right.
\end{equation}
 Moreover the following choices, which are unique up to
 constants, of $u$, $Q$, $R$ and $H$ for $\alpha(z, w,
 \lambda)$ in \eqref{eq:alpha2} give solutions for 
 \eqref{eq:conditionA}:
\begin{equation}\label{eq:solforMaurer}
\left\{
\begin{array}{cr}
u  \in \mathbb R\;, R = - \bar Q\;, H \in i
 \mathbb R^*\;\; &\mbox{for the $(C, 1)$ case,} \\[0.05cm]
u  \in i \mathbb R\;, R =  Q = - \frac{1}{2}
 \bar H\;, H \in \mathbb C^*\;\; &\mbox{for the $(C, 2)$ case,} \\[0.05cm]
u  \in \mathbb R\;, R = \bar Q\;, H \in\mathbb R^*\;\; &\mbox{for
 the $(C, 3)$ case,}\\[0.05cm]
u \in \mathbb R\;, R = \bar Q\;, H \in i \mathbb R^*\;\; &\mbox{for the $(C, 4)$ case,}
\\[0.1cm]
u \in i \mathbb R\;, Q = R = - \frac{1}{2}\bar H, H \in\mathbb C^* \;\;
&\mbox{for the $(S, 1)$ case,} \\[0.05cm]
u \in \mathbb R\;,  Q, R \in i \mathbb R,  H \in i \mathbb R^*\;\; &\mbox{for the $(S, 2)$ case,}\\[0.05cm]
u \in i \mathbb R\;, Q = R = \frac{1}{2} \bar H, H \in \mathbb C^* \;\;
&\mbox{for the $(S, 3)$ case.}

\end{array}
\right.
\end{equation}
 We denote loop groups whose loop algebras are $\Lambda \mathfrak{sl}(2,
 \mathbb C)_{\sigma}^{(\mathfrak{c}, j)}$ and  $\Lambda \mathfrak{sl}(2, \mathbb
 C)_{\sigma}^{(\mathfrak{s}, j)}$ by 
\begin{equation}\label{eq:RealformGroup}
\Lambda  SL(2, \mathbb C)_{\sigma}^{(\mathfrak{c}, j)}\;\mbox{for}\;\; j \in \{1, 2, 3, 4\} \;\;\mbox{and}\;\;
\Lambda SL(2, \mathbb C)_{\sigma}^{(\mathfrak{s}, j)}\;\mbox{for}\;\; j
 \in \{1, 2, 3\}. 
\end{equation}
 If the Maurer-Cartan form $\alpha  =F^{-1}dF$ is in
 $\Omega(\Lambda\mathfrak{sl}(2, \mathbb C)_{\sigma}^{(\mathfrak{c},
 j)})$  for $j \in \{1, 2, 3, 4\}$ (resp. $\Omega(\Lambda \mathfrak{sl}(2, \mathbb
 C)_{\sigma}^{(\mathfrak{s}, j)})$ for $j \in \{1, 2, 3\}$), the corresponding complex
 extended framing  $F$ is in $\Lambda SL(2, \mathbb
 C)_{\sigma}^{(\mathfrak{c}, j)}$ (resp. $\Lambda SL(2, \mathbb
 C)_{\sigma}^{(\mathfrak{s}, j)}$) under the initial condition $F(z_*, w_*,
 \lambda) ={\rm id}$  with $(z_*, w_*) = (z_*, \bar z_*) \in \mathfrak
 D^2$ (resp. $(z_*, w_*) = (\bar z_*, \bar w_*) \in \mathfrak D^2$). We
 denote the complex extended framing $F$ which is a loop in  $\Lambda
 SL(2, \mathbb C)_{\sigma}^{(\mathfrak{c}, j)}$ (resp. $\Lambda SL(2, \mathbb
 C)_{\sigma}^{(\mathfrak{s}, j)}$) by $F^{(\mathfrak{c}, j)}$
 (resp. $F^{(\mathfrak{s}, j)}$).

 We now set the following formulas $\varPhi^{(\mathfrak{c},j)}$ for
 $j \in \{1, 2, 3, 4\}$ (resp. $\varPhi^{(\mathfrak{s},j)}$ for $j
 \in \{1, 2, 3\}$) analogous to the second formula in \eqref{eq:Sym-Bobenko}:
\begin{align}
  \varPhi^{(\mathfrak{c},j)} &= \displaystyle
   \left.-\frac{1}{2|H|}\left( i \lambda \partial_{\lambda}
			 F^{(\mathfrak{c}, j)} (z, \bar z, \lambda)
				    \cdot F^{(\mathfrak{c}, j)} (z, \bar
				    z, \lambda)^{-1}
			\right)\right|_{\lambda \in  S^1} \mbox{for $j
   \in \{1, 2, 3\},$} \label{eq:Sym-Bobenko-2}\\
 \varPhi^{(\mathfrak{c},4)} &=  \frac{1}{2}\left.\left(F^{(\mathfrak{c}, 4)} (z, \bar z,
  \lambda) \left( \begin{smallmatrix} e^{q/2} & 0 \\ 0 &
		  e^{-q/2}\end{smallmatrix}\right) (F^{(\mathfrak{c},
 4)} (z, \bar z, \lambda))^*\right)\right|_{\lambda \in  S^r}\;, \label{eq:Sym-Bobenko-3}\\
  \varPhi^{(\mathfrak{s},j)} &= \displaystyle
 \left.-\frac{1}{2|H|}\left( \lambda \partial_{\lambda} F
 ^{(\mathfrak{s}, j)}(x, y, \lambda)  \cdot F^{(\mathfrak{s}, j)}(x, y,
 \lambda)^{-1} \right)\right|_{\lambda \in  \mathbb R^*} \mbox{for $j
   \in \{1, 2, 3\}$}, \label{eq:Sym-Bobenko-4}
\end{align}
 where $\lambda = \exp (i t) \in S^1$ or $\lambda = \exp (
 q/2+i t) \in S^r$  for \eqref{eq:Sym-Bobenko-2} or
 \eqref{eq:Sym-Bobenko-3} (resp. $\lambda = \pm \exp (t) \in \mathbb R^*$
 for \eqref{eq:Sym-Bobenko-4}) with $t, q \in \mathbb R$, and where $*$
 denotes $X^*= \bar X^t$ for $X \in
 M_{2 \times 2}(\mathbb C)$.
 We note that $w = \bar z$ (resp. $z = \bar z = x \in \mathbb R$ and $w =
 \bar w = y \in \mathbb R$) for $\varPhi^{(\mathfrak{c},j)}$
 (resp. $\varPhi^{(\mathfrak{s},j)}$), from \eqref{eq:coordinates}.
 Then, for each $\lambda \in S^1$ or $\lambda \in S^r$
 (resp. $\lambda \in \mathbb R^*$), the formula
 $\varPhi^{(\mathfrak{c},j)}$ (resp.
 $\varPhi^{(\mathfrak{s},j)}$) defines a map into
 one of the following spaces:
\begin{equation*}
\left\{
\begin{array}{cl}
\mathfrak{su}(1, 1) \cong \mathbb R^{1,2} &\mbox{for the $(C, 1)$ and $(S, 1)$ cases,}\\
 \mathfrak{sl}_*(2, \mathbb R)\cong \mathbb R^{1,2} &\mbox{for the $(C, 2)$ and $(S, 2)$
 cases,}\\
\mathfrak{su}(2) \cong \mathbb R^3 &\mbox{for the $(C, 3)$ and $(S, 3)$
 cases,}\\
SL(2, \mathbb C)/SU(2) \cong H^3 &\mbox{for the $(C, 4)$ case,}
\end{array}
\right.
\end{equation*}
 where $\mathfrak{sl}_*(2, \mathbb R) = \{ g \in \mathfrak{sl} (2,
 \mathbb C) \;| \;g = \left(\begin{smallmatrix}a & b\\c & -a
			    \end{smallmatrix}\right), a \in \mathbb
 R,\;b, c \in i \mathbb R\}$, which is isomorphic to $\mathfrak{sl} (2,
 \mathbb R)$. Here $\mathbb R^{1,2}$ and $\mathbb
 R^3$ can be identified with $\mathfrak{su}(1, 1)$, $ \mathfrak{sl}_*(2,
 \mathbb R)$ and $\mathfrak{su}(2)$ analogous to the identification
 \eqref{eq:ident1}.  Minkowski $4$-space $\mathbb R^{3,1}$ can be identified with ${\rm
 Herm}(2):=\left\{ X \in M_{2 \times 2}(\mathbb C) \;|\; \bar X^t = X
 \right\}$ via the map  $$(x_1, x_2, x_3, x_0) \mapsto
 \frac{1}{2}\begin{pmatrix}x_0+x_3 &

	    x_1+i x_2 \\ x_1-i x_2 & x_0-x_3  \end{pmatrix},
 $$
 then $H^3 \subset \mathbb R^{3,1}$ can be identified with ${\rm Herm}(2)$
 with the determinant $1/4$. 
 Then the inner product for  $\mathfrak{su}(1, 1) \cong \mathbb R^{1,2}$
 (resp. $\mathfrak{sl}_*(2, \mathbb R) \cong \mathbb R^{1,2}$ or
 $\mathfrak{su}(2) \cong \mathbb R^{3}$) can be defined by $\langle a,
 b\rangle = -2 {\rm Tr} \;(a b)$ for $a, b \in \mathfrak{su}(1, 1)$
 (resp. $a, b \in \mathfrak{sl}_*(2, \mathbb R)$ or $a, b \in
 \mathfrak{su}(2)$). The inner product for ${\rm Herm}(2) \cong \mathbb R^{3,1}$ can
 be defined by $\langle a, b\rangle = -2 {\rm Tr}\; (a \sigma_2 b^t \sigma_2)$
 for $a, b \in {\rm Herm}(2)$, where $\sigma_2$ is defined in \eqref{eq:ident2}.
 From now on, we always assume that the
 spectral parameter $\lambda$ is in $S^1$ or $S^r$ for the almost
 compact cases and $\lambda$  is in $\mathbb R^*$ for the almost split
 cases, respectively.
\begin{Remark}
 For the $(C, 4)$ case, the complex Gau{\ss} equation in \eqref{eq:GC} can be reduced to
 the elliptic cosh-Gordon type equation by the choices of functions in
 \eqref{eq:solforMaurer}. It is known that the Gau{\ss} equation for
 {\sc CMC} surfaces with mean curvature $|H| <1$ in
 $H^3$ is the elliptic cosh-Gordon type equation, see
 \cite{BB:MiniH3}. Therefore it is
 natural to use the Sym formula defined in \eqref{eq:Sym-Bobenko-3} for
 the $(C, 4)$ case. 
\end{Remark}
 We denote the metrics for $\varPhi^{(\mathfrak{c}, j)}$ by
 $g^{(\mathfrak{c}, j)}$  (resp. $\varPhi^{(\mathfrak{s}, j)}$ by $g^{(\mathfrak{s}, j)}$), and
 also denote the coefficient matrices for the metrics $g^{(\mathfrak{c}, j)}$
  by $\tilde I^{(\mathfrak{c}, j)}$  for
 $j \in \{1, 2, 3, 4\}$ (resp. $g^{(\mathfrak{s}, j)}$ by $\tilde I^{(\mathfrak{s}, j)}$ for $j \in \{1, 2, 3\}$).
 
 Since $\lambda \in S^1$ or $S^r$ for the almost compact cases and $\lambda
 \in \mathbb R^*$ for the almost split cases, $\tilde
 I^{(\mathfrak{c}, j)}$ and $\tilde I^{(\mathfrak{s}, j)}$ are given as follows:
 \begin{equation}\label{eq:firstforreal}
\left\{
\begin{array}{ll}
 \tilde I^{(\mathfrak{c}, j)} = 
   {\displaystyle \frac{1}{2 |H|^2}} \begin{pmatrix} 
      \mathfrak{a} + \mathfrak{b} + \mathfrak{c} & i( \mathfrak{b} -\mathfrak{c}) \\
     i( \mathfrak{b} -\mathfrak{c}) & \mathfrak{a} -\mathfrak{b}-\mathfrak{c}
    \end{pmatrix}\;\;\mbox{for $j \in \{1, 2, 3\}$}, 
\\[0.5cm]
 \tilde I^{(\mathfrak{c}, 4)} = 
   \displaystyle - H^2  e^{u} \cosh^2 (q) \begin{pmatrix} 1  & 0 \\ 0 & 1 \end{pmatrix}, 
\\[0.5cm]
 \tilde I^{(\mathfrak{s}, j)} = 
   {\displaystyle \frac{1}{2 |H|^2}}
    \begin{pmatrix} 
     -\mathfrak{b}&-\frac{1}{2} \mathfrak{a}  \\
     -\frac{1}{2} \mathfrak{a} &  -\mathfrak{c}
    \end{pmatrix}\;\;\mbox{for $j \in \{1, 2, 3\}$} ,

\end{array}
\right. 
\end{equation}
 where $\mathfrak{a} = H^2 e^u/2 + 2 Q R e^{-u}$,
 $\mathfrak{b} = - \lambda^{-2} H Q$, $\mathfrak{c} = - \lambda^2 H
 R$ and $u$, $Q$, $R$ and $H$ are solutions defined in
 \eqref{eq:solforMaurer}. Therefore it is easy to verify that the
 determinants of $ \tilde I^{(\mathfrak{c}, j)}$ and $ \tilde
 I^{(\mathfrak{s}, j)}$ are as follows:
\begin{equation}\label{eq:firstforreal2}
 \left\{
 \begin{array}{l}
 \det \widetilde{I}^{(\mathfrak{c}, j)} =\displaystyle \frac{1}{|H|^4} \left(\frac{1}{4}H^2
  e^u - Q R e^{-u}\right)^2 \;\;\mbox{for $j \in \{1, 2, 3 \}$, }
 \\[0.3cm]   
 \det \widetilde{I}^{(\mathfrak{c}, 4)} =\displaystyle H^4 e^{2 u} \cosh^4 (q),\\[0.1cm] 
 \det \widetilde{I}^{(\mathfrak{s}, j)} =\displaystyle -\frac{1}{4|H|^4} \left(\frac{1}{4}H^2
  e^u - Q R e^{-u}\right)^2 \;\;\mbox{for $j \in \{1, 2, 3 \}$}.
\end{array}
\right.
\end{equation}
 From \eqref{eq:firstforreal2} one can verify that
 $\varPhi^{(\mathfrak{c},j)}$ and $\varPhi^{(\mathfrak{s},j)}$
 actually define immersions if and only if $e^{u} \neq 4 H^{-2} Q R$ for
 $j \in \{1, 2, 3\}$, and for the $(C, 4)$ case,
 $\varPhi^{(\mathfrak{c},4)}$ always defines an immersion.
 From \eqref{eq:solforMaurer}, the immersions
 $\varPhi^{(\mathfrak{c}, j)}$  for $j \in \{1, 3, 4\}$ and
 $\varPhi^{(\mathfrak{s}, j)}$ for $j \in \{1, 3\}$ are spacelike, and
 the immersions $\varPhi^{(\mathfrak{c}, 2)}$ and
 $\varPhi^{(\mathfrak{s}, 2)}$ are timelike.
 \begin{Remark}
 Since we consider Minkowski space $\mathbb R^{1,2}$ as the three
  dimensional vector space $\{ (x_1, x_2, x_3)
 \;|\;  x_j \in \mathbb R\}$ endowed with the metric $g = dx_1^2 - dx_2^2
 -dx_3^2$, the spacelike, timelike and lightlike vectors are $\langle a, a
  \rangle < 0$, $\langle b, b \rangle > 0$ and $\langle c, c \rangle = 0$ for $a, b, c \in \mathbb
 R^{1,2}$, respectively.
\end{Remark}

 Let $N^{(\mathfrak{c}, j)}$ and $N^{(\mathfrak{s},  j)}$ be the following maps:
 \begin{equation}\label{eq:Gaussmap}
\left\{
 \begin{array}{l}
 N^{(\mathfrak{c}, j)} := \frac{\ell}{2}{\rm Ad} (F^{(\mathfrak{c}, j)})\sigma_3,
  \;\mbox{for}\;\; j \in  \{1, 2, 3\}\;,\\[0.25cm]
 N^{(\mathfrak{c}, 4)} := \frac{1}{2}F^{(\mathfrak{c},
 4)}\left(\begin{smallmatrix} e^{q/2} & 0 \\ 0 & -e^{-q/2}
	  \end{smallmatrix} \right) 
 (F^{(\mathfrak{c},
 4)})^*
  \;\;,\\[0.25cm]
 N^{(\mathfrak{s},  j)} := \frac{\ell}{2}{\rm Ad} (F^{(\mathfrak{s}, j)})
 \sigma_3\;\;\mbox{for}\;\; j \in  \{1, 2, 3\}\;,
\end{array}
\right.
 \end{equation}
 where $\ell$ is $i$ (resp. 1) for $j \in \{1, 3\}$ (resp. $j = 2$).
 It is clear that $N^{(\mathfrak{c}, j)}$ and $
 N^{(\mathfrak{s}, j)} $ are the Gau{\ss} maps of the immersions
 $\varPhi^{(\mathfrak{c}, j)}$ and $\varPhi^{(\mathfrak{s}, j)}$
 respectively. 
 The second fundamental forms
 $I\!I^{(\mathfrak{c}, j)}$ and $I\!I^{(\mathfrak{s}, j)}$
 for the immersions $\varPhi^{(\mathfrak{c}, j)}$ and
 $\varPhi^{(\mathfrak{s}, j)}$ are defined by (see \cite[page
 107]{Oneill:Semi-Riemannian})
\begin{equation*}
\left\{
\begin{array}{l}
  I\!I^{(\mathfrak{c}, j)} = - \langle d
   \varPhi^{(\mathfrak{c}, j)}, d N^{(\mathfrak{c}, j)}\rangle
   \;\;\mbox{for}\;j\in \{1, 2, 3, 4\},\\[0.2cm]
 I\!I^{(\mathfrak{s}, j)} = - \langle d
\varPhi^{(\mathfrak{s}, j)}, d N^{(\mathfrak{s},
j)}\rangle\;\;\mbox{for}\; j \in \{1, 2, 3\}.
\end{array}
\right.
\end{equation*}
  We denote the coefficient matrices of $I\!I^{(\mathfrak{c}, j)}$
  by $\widetilde{I\!I}{}^{(\mathfrak{c}, j)}$
  (resp. $I\!I^{(\mathfrak{s}, j)}$ by
  $\widetilde{I\!I}{}^{(\mathfrak{s}, j)}$). A straightforward
  computation (see also the proof of Theorem \ref{thm:Sym-Bob}) shows
  that the $\widetilde {I\!I}{}^{(\mathfrak{c}, j)}$ for $j \in \{1, 2,
  3, 4\}$ and $\widetilde {I\!I}{}^{(\mathfrak{s}, j)}$ for $j \in \{1,
  2, 3\}$ are as follows:
\begin{equation}\label{eq:secondforint}
\left\{
\begin{array}{l}
 \widetilde{I\!I}{}^{(\mathfrak{c}, j)} =\displaystyle -\frac{2i \ell}{|H|}\left(\frac{1}{4} H^2 e^u - Q  R e^{-u}\right) 
  \begin{pmatrix} 1 & 0 \\ 0& 1 \end{pmatrix}
  \; \mbox{for $j \in \{1, 2, 3\}$, } \\[0.5cm]  

 \widetilde{I\!I}{}^{(\mathfrak{c}, 4)} =   \begin{pmatrix} \mathfrak{d} +2 {\rm Re}\; \mathfrak{e} &
                                             -2 {\rm Im}\; \mathfrak{e} \\ -2 {\rm Im}\; \mathfrak{e} 
					     & \mathfrak{d}
					     -2 {\rm Re}\; \mathfrak{e}
					    \end{pmatrix}, \\[0.5cm]  
 \widetilde{I\!I}{}^{(\mathfrak{s}, j)} =   \displaystyle 
 - \frac{\ell}{|H|}\left(\frac{1}{4} H^2 e^u -Q  R e^{-u}\right) 
\begin{pmatrix} 0 & 1 \\ 1& 0 \end{pmatrix} 
 \; \mbox{for $j \in \{1, 2, 3\}$, } 
\end{array}
\right.
\end{equation}
 where $\mathfrak{d} = H^2 e^u \cosh (q) \sinh (q)$ and $\mathfrak{e} =
 H Q \cosh (q)e^{- 2i t}$.

 We recall that the Gau{\ss} curvatures $K^{(\mathfrak{c}, j)}$ and
 $K^{(\mathfrak{s}, j)}$ (resp. the mean curvature $H^{(\mathfrak{c},
 4)}$) of the immersions $\varPhi^{(\mathfrak{c},
 j)}$ and $\varPhi^{(\mathfrak{s}, j)}$ for $j \in \{ 1, 2, 3\}$
 (resp. $\varPhi^{(\mathfrak{c}, 4)}$) are defined as follows (see
 also \cite[page 157, (93)]{Weinstein:Lorentz}):
\begin{equation*}
\left\{
\begin{array}{l}
 K^{(\mathfrak{c}, j)}:= \pm \det  \left(\tilde{I}^{(\mathfrak{c},
				    j)-1}\widetilde{I\!I}{}^{(\mathfrak{c}, j)}\right)
\;\mbox{and}\; K^{(\mathfrak{s}, j)}:= \pm \det  \left(\tilde  I^{(\mathfrak{s},
  j)-1}\widetilde{I\!I}{}^{(\mathfrak{s}, j)}\right) \; \mbox{for} \; j \in \{1, 2, 3\}, \\[0.3cm]

H^{(\mathfrak{c}, 4)} :=  \displaystyle \frac{1}{2} {\rm Tr} \left(\tilde
   I^{(\mathfrak{c}, 4) -1}\widetilde{I\!I}{}^{(\mathfrak{c}, 4)}\right),
\end{array}
\right.
\end{equation*}
 where the plus sign (resp. the minus sign) has been chosen if the
 surface is in $\mathbb R^3$ or timelike in $\mathbb R^{1,2}$,
 i.e., $\varPhi^{(\mathfrak{c}, j)}$ and $\varPhi^{(\mathfrak{s}, j)}$
 for $j \in \{2, 3\}$ (resp. spacelike in
 $\mathbb R^{1, 2}$, i.e., $\varPhi^{(\mathfrak{c}, 1)}$ and
 $\varPhi^{(\mathfrak{s}, 1)}$), see \cite{Weinstein:Lorentz}.
 Combining \eqref{eq:firstforreal2} with \eqref{eq:secondforint}, we
 finally obtain 
\begin{equation*}
\left\{
\begin{array}{l}
 K^{(\mathfrak{s}, 1)} = K^{(\mathfrak{s}, 2)} = K^{(\mathfrak{c}, 3)} =  4 |H|^2 >0,\\[0.1cm]
 K^{(\mathfrak{c}, 1)}= K^{(\mathfrak{c}, 2)}=K^{(\mathfrak{s}, 3)} = -
 4|H|^2 < 0 , \\[0.1cm]
 H^{(\mathfrak{c}, 4)} = \tanh (-q).
\end{array}
\right.
\end{equation*} 
 The above discussion is summarized in the following theorem:
\begin{Theorem}\label{thm:compactsplit}
  Let $F(z, w, \lambda)$ be the complex extended framing of some complex
  {\sc CGC}-immersion $\varPhi$. Then the following statements hold:
\begin{enumerate}
\item[$(C, 1)$] If $F^{-1} d F$ is in $\Omega(\Lambda \mathfrak{sl}(2,
	      \mathbb C)_{\sigma}^{(\mathfrak{c}, 1)})$, then for each $\lambda \in S^1$ the Sym formula in
	      \eqref{eq:Sym-Bobenko-2} defines a {\rm spacelike
	      constant negative {Gau\ss ian} curvature surface} in
	      $\mathbb R^{1,2}$.

\item[$(C, 2)$] If $F^{-1} d F$ is in $\Omega(\Lambda \mathfrak{sl}(2,
	      \mathbb C)_{\sigma}^{(\mathfrak{c}, 2)})$, then for each $\lambda \in S^1$ the Sym formula in
	      \eqref{eq:Sym-Bobenko-2} defines
	      a {\rm timelike constant negative Gau{\ss}ian curvature
	      surface} in $\mathbb R^{1,2}$.

\item[$(C, 3)$] If $F^{-1} d F$ is in $\Omega(\Lambda \mathfrak{sl}(2,
	      \mathbb C)_{\sigma}^{(\mathfrak{c}, 3)})$, then for each $\lambda \in
	      S^1$ the Sym formula in
	      \eqref{eq:Sym-Bobenko-2} defines a
	      {\rm constant positive {Gau\ss ian} curvature surface} in
	      $\mathbb R^3$.

\item[$(C, 4)$] If $F^{-1} d F$ is in $\Omega(\Lambda \mathfrak{sl}(2,
	      \mathbb C)_{\sigma}^{(\mathfrak{c}, 4)})$, then for each
	      $\lambda \in S^r$ the Sym formula in
	      \eqref{eq:Sym-Bobenko-3} defines
	      a {\rm constant mean curvature surface} with mean
		curvature $|H^{(\mathfrak{c}, 4)}| < 1$ in $H^3$.

\item[$(S, 1)$] If $F^{-1} d F$ is in $\Omega(\Lambda \mathfrak{sl}(2,
	      \mathbb C)_{\sigma}^{(\mathfrak{s}, 1)})$, then for each $\lambda \in \mathbb R^*$ the Sym formula in
	      \eqref{eq:Sym-Bobenko-4} defines a {\rm spacelike
	      constant positive Gau{\ss}ian curvature surface} in
	      $\mathbb R^{1,2}$.

\item[$(S, 2)$] If $F^{-1} d F$ is in $\Omega(\Lambda \mathfrak{sl}(2,
	      \mathbb C)_{\sigma}^{(\mathfrak{s}, 2)})$, then for each $\lambda \in \mathbb R^*$ the Sym formula in
	      \eqref{eq:Sym-Bobenko-4} defines a
	      {\rm timelike constant positive Gau{\ss}ian curvature surface} in
	      $\mathbb R^{1,2}$.

\item[$(S, 3)$] If $F^{-1} d F$ is in $\Omega(\Lambda \mathfrak{sl}(2,
	      \mathbb C)_{\sigma}^{(\mathfrak{s}, 3)})$, then for each $\lambda \in \mathbb R^*$ the Sym formula in
	      \eqref{eq:Sym-Bobenko-4} defines a {\rm constant negative
	      Gau{\ss}ian curvature surface} in $\mathbb R^3$.
\end{enumerate}
\end{Theorem}
 \begin{Definition}
  Let $F^{(\mathfrak c, j)}(z, \bar z, \lambda)$ for $j \in \{1, 2, 3, 4\}$
  (resp. $F^{(\mathfrak s, j)}(x, y, \lambda)$ for $j \in \{1, 2, 3\}$) be
  the complex extended framings, which are elements in ${\Lambda SL(2,
  \mathbb  C)_\sigma}^{(\mathfrak c, j)}$ (resp. ${\Lambda SL(2, \mathbb
  C)_\sigma}^{(\mathfrak s, j)}$). Then $F^{(\mathfrak c, j)}(z, w,
  \lambda)$ (resp. $F^{(\mathfrak s, j)}(x,  y, \lambda)$) is called the
  {\rm extended framing for the immersion}
  $\varPhi^{(\mathfrak{c}, j)}$ (resp. $\varPhi^{(\mathfrak{s},  j)}$).
 \end{Definition}
%
%
 It is known that for three classes of surfaces in the above seven
 classes, there exist parallel constant mean curvature
 surfaces in $\mathbb R^3$ or $\mathbb R^{1,2}$, see also
\cite{Inoguchi:timelike} and \cite{Inoguchi:Minkowski}.
\begin{Corollary}\label{coro:compactsplit} We retain the assumptions in Theorem
 \ref{thm:compactsplit}. Then we have the following:
\begin{enumerate}
\item[$(C, 1M)$] For the $(C, 1)$ case in Theorem \ref{thm:compactsplit},
	       there exists a parallel spacelike constant mean
	       curvature surface with mean curvature $H^{(\mathfrak{c},
		 1)} = |H|>0$ in $\mathbb R^{1,2}$.

\item[$(C, 3M)$] For the $(C, 3)$ case in Theorem \ref{thm:compactsplit},
	       there exists a parallel constant mean curvature surface
		 with mean curvature $H^{(\mathfrak{c}, 3)} =
		 |H|>0$ in $\mathbb R^3$.

\item[$(S, 2M)$] For the $(S, 2)$ case in Theorem \ref{thm:compactsplit},
	       there exists a parallel timelike constant mean curvature
	       surface with mean curvature $H^{(\mathfrak{s}, 2)} = |H|>0$ in $\mathbb R^{1,2}$.
\end{enumerate}
\end{Corollary}
\begin{proof}
 Let $\varPhi^{(\mathfrak{c}, 1)}$, $\varPhi^{(\mathfrak{c},
 3)}$ and $\varPhi^{(\mathfrak{s}, 2)}$ be a spacelike constant
 negative Gau{\ss}ian curvature surface
 in $\mathbb R^{1,2}$, a constant positive Gau{\ss}ian curvature surface
 in $\mathbb R^3$ and a timelike constant positive Gau{\ss}ian curvature
 surface in $\mathbb R^{1,2}$, as defined in Theorem \ref{thm:compactsplit},
 respectively. Let $N^{(\mathfrak{c}, 1)}$, $N^{(\mathfrak{c},
 3)}$ and $N^{(\mathfrak{s},
 2)}$ be the Gau{\ss} maps for $\varPhi^{(\mathfrak{c}, 1)}$, $\varPhi^{(\mathfrak{c},
 3)}$ and $\varPhi^{(\mathfrak{s}, 2)}$ defined in \eqref{eq:Gaussmap}, respectively.
 Then the parallel surfaces for $\varPhi^{(\mathfrak{c},
 1)}$, $\varPhi^{(\mathfrak{c}, 3)}$ and
 $\varPhi^{(\mathfrak{s}, 2)}$ are defined by
 \begin{equation*}
\left\{
\begin{array}{l}
 \varPsi^{(\mathfrak{c}, j)}  := \varPhi^{(\mathfrak{c}, j)} +
  \frac{1}{2 |H|} N^{(\mathfrak{c}, j)} \;\;\;\mbox{for $j \in \{1, 3\}$,}\\[0.2cm]
  \varPsi^{(\mathfrak{s}, 2)}  := \varPhi^{(\mathfrak{s}, 2)} +
  \frac{1}{2 |H|} N^{(\mathfrak{s}, 2)}.
\end{array}
\right.
\end{equation*}
 Then the first fundamental forms and the second
 fundamental forms for these immersions can be computed explicitly, and
 we can easily show that these immersions $\varPsi^{(\mathfrak{c}, 1)}$,
 $\varPsi^{(\mathfrak{c}, 3)}$ and $\varPsi^{(\mathfrak{s}, 2)}$
 define a spacelike constant mean curvature surface with mean curvature
 $H^{(\mathfrak c, 1)} =|H|$ in $\mathbb R^{1,2}$, a constant mean
 curvature surface with mean curvature $H^{(\mathfrak c, 3)} =|H|$ in
 $\mathbb R^3$ and a timelike constant mean curvature surface with mean
 curvature $H^{(\mathfrak s, 2)} =|H|$ in $\mathbb R^{1,2}$,
 respectively. This completes the proof.
\end{proof}
\begin{Definition}
 The surfaces defined in Theorem \ref{thm:compactsplit} and
 Corollary \ref{coro:compactsplit} are called the {\rm integrable surfaces}.
\end{Definition}
\begin{Remark}
 For the three classes of surfaces in Theorem \ref{thm:compactsplit}, which
 are spacelike constant positive Gau{\ss}ian curvature surfaces
 in $\mathbb R^{1, 2}$, constant negative Gau{\ss}ian curvature surfaces in $\mathbb R^3$ and
 timelike constant negative Gau{\ss}ian curvature surfaces in $\mathbb
 R^{1,2}$, there never exist parallel constant mean curvature
 surfaces.
\end{Remark}

\subsection{Gau{\ss} maps of integrable surfaces}\label{subsc:gaussmap}
 In this subsection, we consider the Gau{\ss} maps of integrable
 surfaces defined in the previous section.

 From \cite{DKP:Complex}, it is known that the complex Gau{\ss} map $N$
  of a complex {\sc CMC}-immersion $\varPsi$ with null coordinates
 ($N$ is also the complex Gau{\ss} map of the parallel complex {\sc
 CGC}-immersion) satisfies the following equation:
 \begin{equation}\label{eq:normal}
  N_{z w} = \rho N \;,
 \end{equation}
 where $(z, w) \in \mathfrak D^2 \subset \mathbb C^2$ 
 and the function $\rho : \mathfrak D^2 \to \mathbb C$ is defined by
 $\rho \cdot i\sigma_3 = [\alpha_{-1}, [\alpha_{1}, i \sigma_3]]$ with
 $\alpha_{j}$ as defined in \eqref{eq:alpha2}.

 From Theorem \ref{thm:Sym-Bob}, we note that the complex Gau{\ss} 
 map $N$ is represented by $ N = \frac{i}{2}{\rm Ad} (F)\; \sigma_3,$
 where $F$ is the complex extended framing of the complex {\sc
 CMC}-immersion $\varPsi$.
  Let $F^{(\mathfrak{c}, j)}$ (resp. $F^{(\mathfrak{s}, j)}$) be the
  extended framing of $\varPhi^{(\mathfrak{c}, j)}$ for $j \in \{1, 2, 3, 4\}$
 (resp. $\varPhi^{(\mathfrak{s}, j)}$ for $j \in \{1, 2, 3 \}$).
  Using \eqref{eq:Gaussmap}, we can easily verify that $N^{(\mathfrak{c}, j)}$
  and $N^{(\mathfrak{s}, j)}$ are maps into the following spaces:
 \begin{equation*}
 \left\{
 \begin{array}{l}
 H^2 = SU(1, 1)/U(1) \;\;\mbox{for}\;\; j =1, \\[0.1cm]
 S^{1, 1} = SL_*(2, \mathbb R)/K\;\;\mbox{for}\;\; j =2, \\[0.1cm]
 S^2 = SU(2)/U(1)\;\;\mbox{for}\;\; j =3, \\[0.1cm]
 SL(2, \mathbb C)/U(1) \;\;\mbox{for}\;\; j =4,
 \end{array}
 \right.
 \end{equation*}
  where $K = \left\{{\rm diag}[a, a^{-1}] \;|\; a \in \mathbb R^* \right\}$, 
  which   is isomorphic to $\mathbb R^*$, and $SL_*(2, \mathbb R) =\{ g \in
  SL(2, \mathbb C)\;|\;g =\left(\begin{smallmatrix} a & b\\ c&
				d\end{smallmatrix}\right), a, d \in 
  \mathbb R, b, c \in i \mathbb R \}$, which is isomorphic to $SL(2,
  \mathbb R)$. It is known that the space $SL(2, \mathbb C)/U(1)$ is a $4$-symmetric
   space via the fourth order automorphism 
  $$
   X \mapsto {\rm Ad}\begin{pmatrix}1/\sqrt{i} & 0 \\ 0& \sqrt{i}
		     \end{pmatrix} \left(\bar X^{t}\right)^{-1}\;\; \mbox{for}\;\; X \in SL(2, \mathbb C).
  $$ The choices of coordinates in
   \eqref{eq:coordinates}, the functions in \eqref{eq:solforMaurer} and
   the relation $\rho i \sigma_3 =[\alpha_{-1}[\alpha_1, i \sigma_3]]$ imply
 \begin{equation}\label{eq:Gaussforint}
 \left(N^{(\mathfrak{c}, j)}\right)_{z \bar z} = \rho^{(\mathfrak{c}, j)} N^{(\mathfrak{c}, j)}
  \;\;\mbox{and}\;\;  \left(N^{(\mathfrak{s},j)}\right)_{x y} =
  \rho^{(\mathfrak{s}, j)}N^{(\mathfrak{s},j)}\;\;,
 \end{equation}
 where $\rho^{(\mathfrak{c}, j)} : \mathfrak D \subset \mathbb C \to \mathbb R$
 and $\rho^{(\mathfrak{s}, j)} : \mathfrak D \subset \mathbb R^2 \to
 \mathbb R$.

 It is well known that the equations in \eqref{eq:Gaussforint} for $j
 \in \{1, 2, 3 \}$ are equivalent to the harmonicities (resp. Lorentz
 harmonicities) of Gau{\ss} maps $N^{(\mathfrak{c},
 j)}$ (resp. $ N^{(\mathfrak{s}, j)}$)  with respect to the second fundamental
 forms defined in the first equations of \eqref{eq:secondforint}
 (resp. third equations of \eqref{eq:secondforint}), see Theorem 13 in 
 \cite{Klotz:Harm-Mink}. 
 We then have the following theorem:
\begin{Theorem}\label{thm:Gaussmap}
 Let $\varPhi^{(\mathfrak{c}, j)}$  for $j \in \{1, 2, 3, 4\}$
 and $\varPhi^{(\mathfrak{s}, j)}$ for $j \in \{1, 2, 3\}$  be
 the integrable surfaces defined in Theorem \ref{thm:compactsplit}
 respectively. Moreover, let $N^{(\mathfrak{c}, j)}$ and
 $N^{(\mathfrak{s}, j)}$ be their Gau{\ss} maps respectively. 
 Then the Gau{\ss} maps $N^{(\mathfrak{c}, j)}$ and
 $N^{(\mathfrak{s}, j)}$ are characterized as follows:
\begin{enumerate}
\item[$(C, S, 1)$] The Gau{\ss} map $N^{(\mathfrak{c}, 1)}$
		 (resp. $N^{(\mathfrak{s}, 1)}$) is a harmonic
		 (resp. Lorentz harmonic) map into $H^2$.

\item[$(C, S, 2)$] The Gau{\ss} map $N^{(\mathfrak{c}, 2)}$
		 (resp. $N^{(\mathfrak{s}, 2)}$) is a harmonic
		 (resp. Lorentz harmonic) map into $S^{1, 1}$.

\item[$(C, S, 3)$] The Gau{\ss} map $N^{(\mathfrak{c}, 3)}$
		 (resp. $N^{(\mathfrak{s}, 3)}$) is a harmonic
		 (resp. Lorentz harmonic) map into $S^{2}$.

\item[$(C, 4)$\;\;\;\;] The Gau{\ss} map $N^{(\mathfrak{c}, 4)}$ is a
		 harmonic map into $SL(2, \mathbb C)/U(1)$.
\end{enumerate}
\end{Theorem}
\begin{proof}
 Let us show the $(C, 4)$ case. 
 We recall that a map from a Riemann surface into a $k$-symmetric space
 $N = G/K$ is harmonic (see, for example, \cite[page
 242]{BP:Adler}) if
 \begin{equation*}
 \begin{array}{l}
 [\alpha_{\mathfrak m}^{\prime} \wedge \alpha_{\mathfrak m}^{\prime \prime}]_{\mathfrak m} =0,\\[0.1cm]
 d \alpha_{\lambda} +\frac{1}{2}[\alpha_{\lambda} \wedge  \alpha_{\lambda}] =0,
 \end{array}
 \end{equation*}
 where  $\alpha_{\lambda} = \lambda^{-1} \alpha_{\mathfrak m}^{\prime} +
 \alpha_{\mathfrak k} + \lambda \alpha_{\mathfrak m}^{\prime \prime}$,
 $\mathfrak g = \mathfrak{k} \oplus \mathfrak{m}$ is the reductive
 decomposition and ${}^{\prime}$
 (resp. ${}^{\prime \prime}$) denotes the $(1,0)$-part (resp. $(0,1)$-part).
 Since the map $N^{(\mathfrak c, 4)}$ has the lift $F^{(\mathfrak c, 4)}
 : \mathfrak D \to \Lambda SL(2, \mathbb C)_\sigma^{(\mathfrak c, j)}$ which is
 defined from the complex extended framing $F$ with the conditions in
 \eqref{eq:solforMaurer},  the Maurer-Cartan
 form 
 $\alpha_{\lambda} = F^{(\mathfrak c,
 4)-1} d F^{(\mathfrak c, 4)}$ has the form  $\alpha_{\lambda} =
 \lambda^{-1} \alpha_{-1}+ \alpha_0+\lambda \alpha_1$ and satisfies the
 Maurer-Cartan equation $d \alpha_{\lambda} +
 \frac{1}{2}[\alpha_{\lambda} \wedge \alpha_{\lambda}] =0$, 
 see \eqref{eq:alpha} and Lemma \ref{lem:complex-CMC}.
 The conditions in \eqref{eq:solforMaurer} imply $\alpha_0 = \alpha_{\mathfrak k}$,
 $\alpha_{-1} = \alpha_{\mathfrak m}^{\prime}$ and $\alpha_{1} =
 \alpha_{\mathfrak m}^{\prime \prime}$, where $\mathfrak{sl}(2, \mathbb
 C) = \mathfrak{k} \oplus \mathfrak{m}$ is the reductive decomposition
 associated to $SL(2, \mathbb C)/U(1)$. Moreover, since $\alpha_{-1}$
 and $\alpha_{1}$ have  off-diagonal forms, it follows that 
 $[\alpha_{\mathfrak m}^{\prime} \wedge \alpha_{\mathfrak m}^{\prime
 \prime}]_{\mathfrak m} =0$. Therefore $N^{(\mathfrak c, 4)}$ is a
 harmonic map into the $4$-symmetric space $SL(2, \mathbb C)/U(1)$.  For
 other cases, since the target spaces are symmetric spaces,
 the condition $[\alpha_{\mathfrak m}^{\prime} \wedge \alpha_{\mathfrak m}^{\prime
 \prime}]_{\mathfrak m} =0$ is vacuous. Thus the Maurer-Cartan equation
 $d \alpha_{\lambda} + \frac{1}{2}[\alpha_{\lambda} \wedge
 \alpha_{\lambda}] =0$ with $\alpha_{\lambda} = \lambda^{-1}
 \alpha_{-1}+ \alpha_0+\lambda \alpha_1$ is equivalent to the map
 being harmonic or Lorentz harmonic. This completes the proof.
\end{proof}

\begin{Remark}
 In fact, in the $(C, 4)$ case, the harmonic map $N^{(\mathfrak{c}, 4)}$
 into the $4$-symmetric space $SL(2, \mathbb C)/U(1)$ is known as the
 so-called {\rm Legendre harmonic map} \cite{Ishihara:GG}.  We will
 discuss this topic in a separate publication \cite{DIK:H3}.
\end{Remark}

\begin{table}
\extrarowheight=1mm
\begin{tabular}{|c|c|c|c|}
\hline
{\small Surfaces class} & {\small Gau{\ss} curvature} & {\small Gau{\ss} 
 curvature} &{\small Parallel {\sc CMC}}\\[1mm]\hline
{\small Surfaces in $\mathbb R^3$ }& {\small $K^{(\mathfrak{s}, 3)} = -4 |H|^2$}& {\small $K^{(\mathfrak{c}, 3)}= 4|H|^2$}& $H^{(\mathfrak c,
 3)}= |H|$ \\[1mm] \hline
{\small Spacelike surfaces in $\mathbb R^{1,2}$} & {\small $K^{(\mathfrak{s}, 1)} =  4
 |H|^2$}& {\small $K^{(\mathfrak{c}, 1)}=-4 |H|^2\;$}&$H^{(\mathfrak c, 1)} = |H|$ \\[1mm] \hline
{\small Timelike surfaces in $\mathbb R^{1,2}$} & {\small $K^{(\mathfrak{c}, 2)} = -4 |H|^2$} &
 {\small $K^{(\mathfrak{s}, 2)}=4 |H|^2$} & $H^{(\mathfrak s,
 2)} = |H|$\\ \hline
{\small Surfaces in $H^3$} &  & & $H^{(\mathfrak c, 4)} = \tanh
 (-q)$ \\[1mm] \hline
\end{tabular}
\caption{Integrable surfaces defined by the real forms of $\Lambda \mathfrak{sl}(2, \mathbb C)_{\sigma}$}
\end{table}
\section{The generalized Weierstra{\ss} type representation for integrable
 surfaces}\label{sc:DPW} The generalized Weierstra{\ss} type
 representation for complex {\sc CMC}-immersions, or equivalently {\sc
 CGC}-immersions as the parallel immersions,
 is the procedure of a
 construction of complex {\sc CMC}-immersions from
 a pair of holomorphic potentials, see \cite{DKP:Complex}. In the
 previous section, we classified all integrable surfaces according to the
 classification of  real forms of $\Lambda \mathfrak{sl}(2, \mathbb
 C)_{\sigma}$.  In this section, we show how all integrable surfaces
 are obtained from the pairs of holomorphic potentials in the generalized
 Weierstra{\ss} type representation.

\subsection{Integrable surfaces via the generalized Weierstra{\ss} type
  representation}\label{subsc:DPW}
 The generalized Weierstra{\ss} type representation for complex {\sc
 CMC}-immersions,  or equivalently {\sc CGC}-immersions as the parallel
 immersions,
 is divided into the following 4 steps, see also
 \cite{DKP:Complex} for more details:
 \begin{description}
 \item[Step 1]  Let  $\Check{\eta} =
 (\eta (z, \lambda), \tau(w, \lambda))$ be a pair of holomorphic
	    potentials of the following forms:
 \begin{equation}
 \label{eq:eqforcheketa}
 \Check{\eta} = (\eta (z, \lambda),\;\; \tau (w, \lambda)) =
 \left(\sum_{k=-1}^{\infty} \eta_k (z) \lambda^{k}, \;\; \sum_{m=-\infty}^{1} \tau_m (w) \lambda^m \right)\;\;,
 \end{equation}
 where $(z, w) \in \mathfrak D^2$ and where $\mathfrak D^2$ is some
	    holomorphically convex domain in $\mathbb C^2$, $\lambda \in
	    \mathbb C^\ast$,  $|\lambda| = r$ $(0 < r < 1)$, and $\eta_k$ and
	    $\tau_m$ are $\mathfrak{sl}(2, \mathbb C)$-valued
	    holomorphic differential 1-forms. Moreover $\eta_k (z)$  and
	    $\tau_k (w)$ are diagonal (resp. off-diagonal) matrices if $k$ is even
	    (resp. odd). We also assume that the upper right entry
	    of $\eta_{-1} (z)$ and the lower left entry $\tau_{1} (w)$
	    do not vanish for all $(z, w) \in \mathfrak D^2$.

 \item[Step 2]  Let $C$ and $L$ denote
 the solutions to the following linear ordinary differential equations
 \begin{equation}\label{eq:eqforC^3}
 d C = C \eta \;\;\mbox{and}\;\;d L = L \tau\;\;\mbox{with}\;\;C(z_*, \lambda) = L(w_*, \lambda) = {\rm id},
 \end{equation}
 where $(z_*, w_*) \in \mathfrak D^2$ is a fixed base point.

\item[Step 3] We factorize the pair of matrices $(C, L)$ via the
	    generalized Iwasawa decomposition of Theorem
	    \ref{doublesplitting} as follows:
 \begin{equation}
 \label{eq:splittingCR}
 (C,\;\; L) = (F,\;\; F) ({\rm id} ,\;\;W) (V_+,\;\; V_-)\;\;,
 \end{equation}
 where $V_{\pm} \in \Lambda^{\pm} SL(2, \mathbb C)_{\sigma}$.
\end{description}
\begin{Theorem}[\cite{DKP:Complex}]\label{thm:DKP-Extendedframings}
 Let $F$ be a $\Lambda SL(2, \mathbb C)_\sigma$-loop defined by the
 generalized Iwasawa decomposition in \eqref{eq:splittingCR}. Then there
 exists a $\lambda$-independent diagonal matrix $l(z, w) \in SL(2,
 \mathbb C)$ such that $F\cdot l$ is a complex extended framing of some
 complex {\sc CMC}-immersion, or equivalently the complex {\sc
 CGC}-immersion as the parallel immersion.
\end{Theorem}
\begin{description}
 \item[Step 4] The Sym formula defined in \eqref{eq:Sym-Bobenko} via
	    $F(z, w, \lambda)l (z, w)$ represents
	    a complex {\sc CMC}-immersion and a complex {\sc CGC}-immersion
	    in $\mathfrak{sl}(2, \mathbb C) \cong \mathbb C^3$.
\end{description}
 Let $\mathfrak{c}_j $ for $j \in \{1, 2, 3, 4\}$ and $\mathfrak{s}_j$
 for $j \in \{1, 2, 3\}$ be the involutions defined in
 \eqref{eq:inv-1-forms}, respectively. Then we define the following pairs of involutions on
 $\check\eta = (\eta, \tau) \in \Omega(\Lambda \mathfrak{sl}(2, \mathbb
 C)_\sigma) \times \Omega(\Lambda \mathfrak{sl}(2, \mathbb C)_\sigma)$:
\begin{equation}\label{eq:involutions}
\mathfrak{r}_j : (\eta, \tau)  \longmapsto (\mathfrak{c}_j \tau, \;\mathfrak{c}_j \eta)
 \;\;\mbox{and}\;\;\mathfrak{d}_j : (\eta, \tau)   \longmapsto (\mathfrak{s}_j \eta, \;\mathfrak{s}_j \tau).
\end{equation}
We now prove the following theorem.
\begin{Theorem}\label{thm:DPWforint}
 Let $\check \eta = (\eta(z, \lambda), \tau (w, \lambda))$ be a pair of
 holomorphic potentials defined as in \eqref{eq:eqforcheketa}, and let
 $\mathfrak{r}_j$ for  $ \;j \in \{1, 2, 3, 4\}$ and $\mathfrak{d}_j$
 for $j \in \{1, 2, 3\}$ be the pairs of involutions defined in
 \eqref{eq:involutions}, respectively. Then the following statements hold:
\begin{enumerate}
\item[$(C, 1)$] If $ \mathfrak{r}_1 (\check \eta) = \check \eta$, then the
	      resulting immersions given by the generalized
	      Weierstra{\ss} type representation are
	      spacelike constant negative Gau{\ss}ian curvature surfaces
	      in $\mathbb R^{1,2}$.

\item[$(C, 2)$] If $\mathfrak{r}_2 (\check \eta) = \check \eta$, then the
	      resulting immersions given by the generalized
	      Weierstra{\ss} type representation are timelike constant
	      negative Gau{\ss}ian curvature surfaces
	      in $\mathbb R^{1,2}$.

\item[$(C, 3)$] If $\mathfrak{r}_3 (\check \eta) = \check \eta$, then the
	      resulting immersions given by the generalized
	      Weierstra{\ss} type representation are
	      constant positive Gau{\ss}ian curvature surfaces
	      in $\mathbb R^{3}$.

\item[$(C, 4)$] If $\mathfrak{r}_4 (\check \eta) = \check \eta$, then the
	      resulting immersions given by the generalized
	      Weierstra{\ss} type representation are
	      constant mean curvature surfaces with mean curvature
		$|H^{(\mathfrak{c}, 4)}| < 1$ in $H^3$.

\item[$(S, 1)$] If $\mathfrak{d}_1 (\check \eta) = \check \eta$, then the
	      resulting immersions given by the generalized
	      Weierstra{\ss} type representation are
	      spacelike constant positive Gau{\ss}ian curvature surfaces
	      in $\mathbb R^{1,2}$.

\item[$(S, 2)$] If $\mathfrak{d}_2 (\check \eta) = \check \eta$, then the
	      resulting immersions given by the generalized
	      Weierstra{\ss} type representation are
	      timelike constant positive Gau{\ss}ian curvature surfaces
              in $\mathbb R^{1,2}$.

\item[$(S, 3)$] If $\mathfrak{d}_3 (\check \eta) = \check \eta$, then the
	      resulting immersions given by the generalized
	      Weierstra{\ss} type representation are
	      constant negative Gau{\ss}ian curvature surfaces in
	      $\mathbb R^{3}$.

\end{enumerate}
 \end{Theorem}
\begin{proof}
 Since the pairs of holomorphic potentials are invariant under the 
 involutions $\mathfrak{r}_j$ or $\mathfrak{d}_j$, the coordinates $(z,
 w) \in \mathfrak D^2$ satisfy the following relations:
 \begin{equation}\label{eq:conjasymp}
 \left\{
  \begin{array}{l}
   w = \bar z \;\;\mbox{if} \;\; \mathfrak{r}_j (\check \eta) = \check \eta, \\
   z = \bar z \;\mbox{and}\; w = \bar w\;\;\mbox{if} \;\; \mathfrak{d}_j (\check \eta) = \check \eta.
  \end{array}
 \right.
 \end{equation}
  Let $(C, L)$ be the pair of solutions of the differential equations in
 \eqref{eq:eqforC^3} with the initial conditions $C(z_*) = L(w_*) = {\rm id}$, 
 where $(z_*, w_*) \in \mathfrak D^2$ satisfies one of the conditions in
 \eqref{eq:conjasymp}. 
  Let $\mathcal R_j$ for $j \in \{1, 2, 3, 4\}$ (resp. $\mathcal D_j$ for
 $j \in \{1, 2, 3\}$) be the following pair of involutions on $\Lambda
 SL(2, \mathbb C)_{\sigma} \times \Lambda SL(2, \mathbb C)_\sigma$:
 \begin{equation}\label{eq:mathcalG}
  \mathcal{R}_j (C, L):= (\mathcal C_j (L), \mathcal C_j (C))
   \;\;\mbox{and} \;\; \mathcal D_j (C, L):= (\mathcal S_j (C), \mathcal
   S_j (L)) \;\;,
 \end{equation}
 where $\mathcal C_j$ (resp. $\mathcal S_j$) are the involutions on
 $\Lambda SL(2, \mathbb C)_{\sigma}$ corresponding to the
 involutions $\mathfrak{c}_j$ (resp. $\mathfrak{s}_j$) as in
 Theorem \ref{thm:almostcompact} (resp. Theorem \ref{thm:almostsplit}), e.g., 
\begin{equation*}
 \mathcal C_1 : C(\lambda) \to \overline{C(-1/\bar
  \lambda)}^{t -1}\;\;\mbox{and}\;\;\mathcal S_1 : C(\lambda) \to
  \overline{C(- \bar \lambda)}^{t -1} \;\;\mbox{for}\;\;
  C(\lambda) \in \Lambda SL(2, \mathbb C)_\sigma.
\end{equation*}
 
 Noting the conditions in \eqref{eq:conjasymp}, the involutions
 $\mathfrak{r}_j$ for $j \in \{1, 2, 3, 4\}$ and $\mathfrak{d}_j$ for $j
 \in \{1, 2, 3\}$ in \eqref{eq:involutions} define symmetries on the
 pair of solutions $(C, L)$ as  follows:
  \begin{equation}\label{eq:symmholoext}
 \left\{
  \begin{array}{l}
   \mathcal R_j (C(z, \lambda), L(\bar z, \lambda)) = (C(z, \lambda), L(\bar z,
    \lambda)) \;\;\mbox{if $\mathfrak{r}_j (\check \eta) = \check \eta$ for} \;\; j \in \{1, 2, 3, 4\}, \\[0.1cm]
   \mathcal D_j (C(x, \lambda), L(y, \lambda)) = (C(x, \lambda), L(y,
    \lambda)) \;\;\mbox{if $\mathfrak{d}_j (\check \eta) = \check \eta$ for}\;\;  j \in \{1, 2, 3\},
  \end{array}
 \right.
 \end{equation}
 where $x = z = \bar z \in \mathbb R$ and $y = w = \bar w \in \mathbb
 R$.  Applying the generalized Iwasawa decomposition of Theorem
 \ref{doublesplitting} for $(C, L) \in \Lambda SL(2, \mathbb C)_{\sigma}
 \times \Lambda SL(2, \mathbb C)_{\sigma}$, we have
 \begin{equation}\label{eq:Iwasawadecom}
 (C, L) = (F, F) ({\rm id}, W) (V_{+}, V_-)\;,
 \end{equation}
 where $V_{\pm} \in \Lambda^{\pm} SL(2, \mathbb C)_{\sigma}$. If $(z, w)
 \in \mathfrak{D}^2$ is sufficiently close to $(z_*, w_*) \in \mathfrak
 D^2$, then the middle term $W$ of the generalized
 Iwasawa decomposition is identity.

 Since the left component $F$ of the generalized Iwasawa
 decomposition in \eqref{eq:Iwasawadecom} can be rephrased as $ F = C
 V_+^{-1} = L V_{-}^{-1}$, we have 
 \begin{equation*}
 C^{-1} L = V_{+}^{-1} V_{-}\;.
 \end{equation*}
 From the symmetries on $(C, L)$ in \eqref{eq:symmholoext}, $V_{-}$ and
 $V_{+}$ have the following relations:
 \begin{equation*}
 \left\{
\begin{array}{l}
 \mathcal
  C_j (V_{\pm}(z, \bar z, \lambda)) =  k^{(\mathfrak{c}, j)}(z,\bar
  z)^{-1}V_{\mp}(z, \bar z, \lambda)\;\;\mbox{if $\mathfrak{r}_j (\check \eta) = \check \eta$ for} \;\; j \in \{1, 2, 3, 4\}, \\[0.2cm]
 \mathcal S_j (V_{\pm}(x, y, \lambda) ) = k^{(\mathfrak{s}, j)}(x,y)^{-1}
 V_{\pm}(x, y, \lambda)  \;\;\mbox{if $\mathfrak{d}_j (\check \eta) = \check \eta$ for}\;\;  j \in \{1, 2, 3\},
\end{array}
 \right.
 \end{equation*}
 where $k^{(\mathfrak{c}, j)}(z,\bar z)$ and $k^{(\mathfrak{s}, j)}(x,
 y)$ are $\lambda$-independent diagonal matrices satisfying the
 symmetries $ \mathcal C_j
 (k^{(\mathfrak{c}, j)}(z, \bar z)) = k^{(\mathfrak{c}, j)}(z, \bar
 z)^{-1}$ and $ \mathcal S_j(k^{(\mathfrak{s}, j)}(x, y)) = k^{(\mathfrak{s}, j)}(x,
 y)^{-1}$ respectively. From the discussion above $F$ has the symmetry as follows:
 \begin{equation}\label{eq:invarianceofF}
 \left\{
  \begin{array}{l}
   \mathcal C_j (F(z, \bar z, \lambda)) = F(z, \bar z, \lambda)k^{(\mathfrak{c},j)} (z,
    \bar z) \;\;\mbox{if} \;\; \mathfrak{r}_j(\check \eta) = \check
    \eta\;\;\mbox{for}\;j \in \{1, 2, 3, 4\}, \\[0.1cm]
   \mathcal S_j (F(x, y, \lambda)) =
   F(x, y, \lambda)k^{(\mathfrak{s}, j)}(x, y)
   \;\;\mbox{if}  \;\;\mathfrak{d}_j(\check \eta) = \check \eta \;\;\mbox{for}\;j \in \{1, 2, 3\}.
  \end{array}
 \right.
 \end{equation}
  Let $F^{(\mathfrak{c}, j)}$
 (resp. $F^{(\mathfrak{s}, j)}$) denote the left components $F$ of the
 generalized Iwasawa decomposition in \eqref{eq:Iwasawadecom} which have
 the symmetries in \eqref{eq:invarianceofF} by $\mathcal C_j$ (resp. $\mathcal S_j$).  
 Let $\tilde k^{(\mathfrak{c}, j)} (z, \bar z)$ and $\tilde
 k^{(\mathfrak{s}, j)}(x, y)$ be the $\lambda$-independent diagonal
 matrices such that $\tilde k^{(\mathfrak{c}, j)} (z, \bar z)^2 =
 k^{(\mathfrak{c}, j)}(z, \bar z)$ and $\tilde k^{(\mathfrak{s}, j)} (x,
 y)^2 = k^{(\mathfrak{s}, j)}(x, y)$, respectively.  
 Setting $\tilde F^{(\mathfrak{c}, j)}(z, \bar z, \lambda) =
 F^{(\mathfrak{c}, j)}(z, \bar z, \lambda) \tilde k^{(\mathfrak{c}, j)}
 (z, \bar z)$ and $\tilde F^{(\mathfrak{s}, j)}(x, y, \lambda) =
 F^{(\mathfrak{s}, j)}(x, y, \lambda) \tilde k^{(\mathfrak{s}, j)}(x,
 y)$, we have 
 \begin{equation*}
\left\{
  \begin{array}{l}
  \mathcal C_j (\tilde F^{(\mathfrak{c}, j)}(z, \bar z, \lambda)) =
   \tilde F^{(\mathfrak{c}, j)}(z, \bar z, \lambda) \;\;\mbox{if } \;\;
   \mathfrak{r}_j   (\check \eta) =\check \eta \;\;\mbox{for}\;j \in \{1, 2, 3, 4\}, \\[0.1cm]
 \mathcal S_j(\tilde F^{(\mathfrak{s}, j)} (x, y, \lambda)) = \tilde
 F^{(\mathfrak{s}, j)} (x, y, \lambda) \;\;\mbox{if} \;\; \mathfrak{d}_j
 (\check \eta) = \check \eta \;\;\mbox{for}\;j \in \{1, 2, 3\}.
  \end{array}
 \right.
 \end{equation*}
 Moreover a straightforward calculation shows that $\alpha^{(\mathfrak{c}, j)} :=\tilde
 F^{(\mathfrak{c}, j)-1} d \tilde F^{(\mathfrak{c}, j)}$ and $\alpha^{(\mathfrak{s}, j)} :=\tilde
 F^{(\mathfrak{s}, j)-1} d \tilde F^{(\mathfrak{s}, j)}$ have the forms in
 \eqref{eq:alpha} with the properties in \eqref{eq:solforMaurer}, i.e., $\tilde
 F^{(\mathfrak{c}, j)} \in \Lambda SL(2, \mathbb
 C)_\sigma^{(\mathfrak{c}, j)}$ and $\tilde F^{(\mathfrak{s}, j)}  \in
 \Lambda SL(2, \mathbb C)_\sigma^{(\mathfrak{s},  j)}$ are the
 extended framings. 
 From the argument in Theorem \ref{thm:compactsplit}, the Sym formulas
 $\varPhi^{(\mathfrak{c}, j)}$ for $j \in \{1, 2,
 3\}$ in \eqref{eq:Sym-Bobenko-2} via $\tilde F^{(\mathfrak{c}, j)}$, 
 $\varPhi^{(\mathfrak{c}, 4)}$ in \eqref{eq:Sym-Bobenko-3} via
 $\tilde F^{(\mathfrak{c}, 4)}$ and $\varPhi^{(\mathfrak{s}, j)}$  for
 $j \in \{1, 2, 3\}$ in \eqref{eq:Sym-Bobenko-4} via $\tilde
 F^{(\mathfrak{s}, j)}$ define immersions
 which have the properties as desired. This completes the proof.
\end{proof}
\begin{Remark}
 From the forms of pairs of involutions $\mathfrak{r}_j$ for $j \in \{1, 2, 3,
 4\}$ defined in  \eqref{eq:involutions}, the pairs of holomorphic
 potentials $\check \eta$ for $(C, j)$ cases in Theorem
 \ref{thm:DPWforint} are generated by a single potential,
 i.e., $\check \eta = (\eta, \tau ) = (\eta, \mathfrak{c}_j ( \eta ))$,
 where $\mathfrak{c}_j$ for $j \in \{1, 2, 3, 4\}$ are involutions defined
 in \eqref{eq:inv-1-forms}.
\end{Remark}
\begin{Remark}
 In the proof of Theorem \ref{thm:DPWforint}, we assume our domain
 $\mathfrak D^2 \subset \mathbb C^2$ is sufficiently small around the
 initial point $(z_*, w_*) \in \mathfrak D^2$ so that the middle term
 $w_n$ of the generalized
 Iwasawa decomposition of Theorem \ref{doublesplitting} is in the
 identity component. In
 general, if we consider the larger domain $\widetilde{\mathfrak
 D}^2$ such that $\mathfrak D^2 \subset \widetilde{\mathfrak D}^2$, then
 the middle terms $w_n$ have many components.  Therefore the extended
 framing $F$ could have singularities on $\widetilde{\mathfrak D}^2$. 
\end{Remark}
\appendix
\section{Basic notation and results for affine Kac-Moody Lie algebras
 and Lie groups}\label{BasicResult}
\subsection{Kac-Moody Lie algebras}\label{subsc:Affine}
 We first recall the definition of the {\it generalized Cartan matrix},
 which is an $n \times n$ matrix $A= (a_{ij})$ satisfying the following
 conditions: 
 \begin{enumerate}
 \item For $i \in \left\{ 1, \dots, n\right\}$, $a_{ii} =2$.
 \item For $i \neq j$,  $a_{ij} \leq 0$.
 \item For $i \neq j$, if $a_{ij} = 0$, then $a_{ji} =0$.
 \end{enumerate}
 The generalized Cartan matrix $A$ is called {\it symmetrizable} if 
 $A$ can be decomposed as $A = D S$, where $D$ is a positive definite
 diagonal matrix and $S$ is a symmetric matrix. The generalized Cartan
 matrix $A =(a_{ij})$ is called {\it decomposable} if there exists
 $\sigma \in S_n$ such that $(a_{\sigma(i)\sigma(j)}) =
 \left(\begin{smallmatrix}B & 0 \\ 0 & C \end{smallmatrix}\right)$,
 where $B$ and $C$ are square matrices. If $A$ is not decomposable, then
 $A$ is called {\it indecomposable}. The symmetrizable and
 indecomposable generalized Cartan matrices are classified into 
 {\it finite}, {\it affine} and {\it indefinite} if $S (= D^{-1} A)$ is
 positive definite, positive semidefinite and indefinite respectively,
 see \cite[Chapter 4]{Kac:Kac-Moody}.
 
 Let $A= (a_{ij})$ be the generalized Cartan matrix of rank $r$, and let
 $\mathfrak h$ be a vector space over the complex field $\mathbb
 C$ such that ${\rm dim} \mathfrak h = n + {\rm corank} (A)$. Moreover, 
 let $\Pi = \left\{  \alpha_1, \dots, \alpha_n \right\}$ and $\Check\Pi =
 \left\{ \check\alpha_1, \dots, \check\alpha_n \right\}$ be linearly
 independent in $\mathfrak h^{*}$ and $\mathfrak h$, respectively, such
 that $\alpha_j (\Check \alpha_i) = a_{ij}$. It follows
 that if $A$ is nonsingular then $\Check \Pi$ (resp. $\Pi$) is a basis
 of $\mathfrak h$ (resp. $\mathfrak h^{*}$).  It is well known that the $\alpha_i$
 (resp. $\check\alpha_i$) for $i \in \{1, 2, \dots, n\}$ are known as
 the {\it roots} (resp. the {\it coroots}).

 We now define the complex Kac-Moody algebra  $\mathfrak g = \mathfrak
 g(A)$ associated to $A$: it is generated by $\left\{ \mathfrak h, e_i,
 f_i; i = 1,2, \dots, n\right\}$ with the following relations
\begin{equation}
\begin{array}{lll}
 [\mathfrak h, \mathfrak h] = 0, &  [e_i, f_j] =
 \delta_{ij}\Check{\alpha}_i & (i, j= 1, 2, \dots, n),\\

 [h, e_i] = \alpha_i ( h ) e_i, & [h, f_i] = -\alpha_i( h ) f_i & (h \in \mathfrak h),\\
({\rm ad} e_i)^{1-a_{ij}}(e_j) = 0, & ({\rm ad} f_i)^{1-a_{ij}}(f_j) = 0
 & (i \neq j).
\end{array}
\end{equation}
 The Kac-Moody algebra $\mathfrak g =\mathfrak g (A)$ is said to be of
 finite, affine or indefinite type if the corresponding generalized
 Cartan matrix $A$ is as well. We note that $\mathfrak{h}$ is known as the 
 Cartan subalgebra of the Kac-Moody algebra $\mathfrak g (A)$.
 
 For a finite dimensional Lie algebra and its Cartan
 matrix, the {\it extended Cartan matrix} can be defined by adding a
 zero'th row and column to the Cartan matrix, corresponding to adding a
 new simple root $\alpha_0 := -\theta$, where $ \theta$ is the maximal root for
 $\mathfrak{g}$ with respect to $\Pi = \left\{  \alpha_1, \dots,
 \alpha_n \right\}$.
 It is known that the extended Cartan matrix is an example of the
 generalized Cartan matrix of affine type \cite{Kac:Kac-Moody}.

\subsection{Loop algebras and loop groups}\label{loopgroups}
 In this subsection, we introduce a loop group and a loop algebra, and  we
 give a characterization of the loop algebras via affine Kac-Moody Lie
 algebras.
 Let $C_r :=\{\lambda \in \mathbb C \;|\;\; |\lambda |=r\}$ be the circle of 
 radius $r$ with $r \in (0, 1]$.
  Let $G$ be a Lie group and let $\mathfrak g$ be
 its Lie algebra.
 For any $r \in (0,1] \subset \mathbb R$, we consider the twisted loop 
 algebra and loop group:
 \begin{equation}\label{eq:loopalgebras}
  \Lambda_r \mathfrak g_\sigma = \left\{ \alpha :C_r \to
  \mathfrak g\; \left| \right. 
  \; \alpha \;\mbox{is continuous and}\; \alpha(-\lambda) = \sigma_3
  \alpha(\lambda)  \sigma_3 \;   \right\} \; ,
 \end{equation}
 \begin{equation}\label{eq:loopgroups}
   \Lambda_{r}G_\sigma = \left\{ g:C_r  \to G\; \left| \right. \; g \;\mbox{is continuous and}\; g(-\lambda) = 
  \sigma_3 g(\lambda) \sigma_3 \; \right \} \; ,
 \end{equation}
 where $\sigma_3$ is defined in \eqref{eq:ident2}.

 Let $\mathcal A$ be the ``Wiener algebra''
 \begin{equation}
 \label{eq:weiner}
 \mathcal A = \left\{ f(\lambda) = \sum_{n \in \mathbb Z} f_n \lambda^n\; 
 :\;  C_r \to \mathbb C \;\; ; \;\; \sum_{n \in \mathbb Z}|f_n| 
 < \infty \right\}\;.
 \end{equation}
 The Wiener algebra is a Banach algebra relative to the norm $\| f\| =
 \sum |f_n|$, and $\mathcal A$ consists of continuous functions.
 Thus the loop groups and loop algebras with coefficients in $\mathcal A$ are 
 Banach Lie groups and Banach Lie algebras. 
 In this paper, we consider only 
 the loop groups and algebras which can be extended
 continuously to $\mathbb C^{*}$.

 From \cite{Kac:Kac-Moody}, we quote the following realization of the
 affine Kac-Moody Lie algebras via loop algebras:
\begin{Theorem}[\cite{Kac:Kac-Moody}]\label{thm:Kac}
 Let $\mathfrak{g}$ be a complex finite dimensional simple Lie algebra,
 and let $A$ be its extended Cartan matrix, and let $\Lambda
 \mathfrak{g}$ be its loop algebra. Moreover, let
 $\tilde{\Lambda}\mathfrak{g} := \Lambda \mathfrak{g} \oplus \mathbb C k
 \oplus \mathbb C c$ be the twice central extension of the loop algebra
 $\Lambda \mathfrak{g}$. Then $\tilde{\Lambda} \mathfrak{g}$ is the 
 affine Kac-Moody Lie algebra associated to the generalized Cartan
 matrix A of affine type.
\end{Theorem}

\subsection{Double loop groups and the generalized Iwasawa
  decompositions}\label{nsc:doubleIwasawa}
 In this subsection, we give the basic notation and results for double
 loop groups, see \cite{DW:CMC-loop} for more details. Let $D_r :=
 \{\lambda \in \mathbb C \;|\;\;  |\lambda| < r\}$ be an open disk.
 Also, let $A_r =\{\lambda \in \mathbb C
 \;|\;\; r < |\lambda| < 1/r\}$ be an open annulus containing $S^1$. 
 Furthermore, let $E_r =\{\lambda \in \mathbb C  \;|\;\; r <
 |\lambda| \} \cup \{\infty\}$  be an exterior of the circle $C_r$.

 We recall the definitions of the {\it twisted plus $r$-loop group}  and
 the {\it minus $r$-loop group} of $\Lambda SL(2, \mathbb C)_\sigma$ as
 follows:
 \begin{equation*}
  \Lambda_{r,  B}^+ SL(2,\mathbb C)_\sigma := \left\{ W_+ \in \Lambda_r
  SL(2,\mathbb C)_\sigma 
  \; \left|  \; 
 \begin{array}{r} 
  W_+(\lambda)  \text{ extends holomorphically}\; \\ 
  \text{to} \;\;D_r  \;\mbox{and}\;W_+(0) \in {\boldsymbol B}. 
 \end{array}
 \right.\right\} \; ,
 \end{equation*}
 \begin{equation*}
  \Lambda_{r, B}^- SL(2,\mathbb C)_\sigma := \left\{ W_- \in \Lambda_r
  SL(2,\mathbb C)_\sigma 
  \; \left| 
 \begin{array}{r}
  W_-(\lambda)  \text{ extends holomorphically}\\
  \text{to}\;\;E_r\; \;\mbox{and}\;W_-(\infty) \in \boldsymbol{B}.
 \end{array}
 \right. \right\} \; , 
 \end{equation*}
 where $\boldsymbol B$ is a subgroup of $SL(2, \mathbb C)$. If
 $\boldsymbol B = \{\rm id\}$ 
 we write the subscript $*$  instead of $\boldsymbol B$, if $\boldsymbol B = 
 SL(2, \mathbb C)$ we  abbreviate  $\Lambda_{r, B}^+ SL(2,\mathbb C)_\sigma$ and  
 $\Lambda_{r, B}^- SL(2,\mathbb C)_\sigma$ by  $\Lambda_r^+ SL(2,\mathbb
 C)_\sigma $ and 
 $\Lambda_r^- SL(2,\mathbb C)_\sigma$, respectively. From now on we will
 use the subscript  $\boldsymbol B$ as above only if $\boldsymbol B \cap
 SU(2) = \{\rm id \}$ holds.  When $r=1$, we always omit the $1$.

 We set the product of two loop groups:
 \begin{equation*}
 \mathcal H = \Lambda_r SL(2, \mathbb C)_\sigma \times \Lambda_R SL(2,
 \mathbb C)_\sigma\;\;,
\end{equation*}
 where $0 < r < R$.  Moreover we set the subgroups of $\mathcal H$ as follows:
\begin{equation*}
\begin{array}{l}
 \mathcal H_+ = \Lambda_r^+ SL(2, \mathbb C)_\sigma \times \Lambda_R^-
 SL(2, \mathbb C)_\sigma, \\[0.2cm]

 \mathcal H_-= \left\{ (g_1,\;\;g_2) \in \mathcal H \left|
 \;\begin{array}{r}
 \mbox{ $g_1$ and $g_2$ extend holomorphically}
 \\\mbox{to $A_r$ and  $g_1 |_{A_r} = g_2 |_{A_r}$
   }\end{array}\right. \right\}\;,
 \end{array}
 \end{equation*}
 We then quote Theorem 2.6 in \cite{DW:CMC-loop}. 
\begin{Theorem}\label{doublesplitting}
 $\mathcal H_- \times \mathcal H_+ \rightarrow \mathcal H_- \mathcal
 H_+$ is an analytic diffeomorphism. The image is open and dense in
 $\mathcal H$. More precisely
 $$
 \mathcal H = \bigcup_{n=0}^{\infty} \mathcal H_- w_n \mathcal H_+\;\;,
 $$
 where $ w_n =   \left({\rm id},\;\;\left(\begin{smallmatrix}\lambda^{n} & 
  0\\ 0 & \lambda^{-n}\end{smallmatrix}\right)\right) $ if $n=2k$ and 
  $  \left( {\rm id},\;\; \left( \begin{smallmatrix} 0 & 
  \lambda^{n} \\ -\lambda^{-n} & 0 \end{smallmatrix}\right)\right)$ if $n=2k+1$.
\end{Theorem}
 The proof of the theorem above is almost verbatim the proof given in
 the  basic decomposition paper \cite{BerG:loop}, see also \cite{DK:cyl}.

\bibliographystyle{plain}
\def\cprime{$'$}

\end{document}